\documentclass[reqno]{amsart}

\makeatletter
\@namedef{subjclassname@2020}{\textup{2020} Mathematics Subject Classification}
\makeatother

\usepackage{amssymb}
\usepackage{graphicx}
\usepackage{mathrsfs}
\usepackage[colorlinks=true, linkcolor=blue, urlcolor=blue, citecolor=blue]{hyperref}
\usepackage{color}
\usepackage{url}
\usepackage{cite}

\usepackage{times}
\usepackage{microtype}
\usepackage{amssymb}
\usepackage{mathtools}
\usepackage{tikz}
\usepackage{wasysym}
\usepackage{centernot}
\usepackage{color}
\usepackage{setspace}
\usepackage{appendix}
\usepackage{booktabs}
\usepackage{multirow}

\newtheorem{theorem}{Theorem}[section]
\newtheorem{lemma}[theorem]{Lemma}

\theoremstyle{definition}
\newtheorem{definition}[theorem]{Definition}

\theoremstyle{remark}
\newtheorem{remark}[theorem]{Remark}

\numberwithin{equation}{section}

\def\D{\mathcal{D}}

\def\bC{\mathbf{C}}

\def\bP{\mathbb{P}}

\def\bZ{{\mathbb Z}}

\def\bR{{\mathbb R}}

\def\sF{{\mathscr F}}
\def\sA{{\mathscr A}}
\def\sG{{\mathscr G}}

\def\sL{\mathscr{L}}
\def\cD{\mathcal{D}}

\def\bE{\mathbb{E}}

\def\bN{\mathbb{N}}

\begin{document}

\title[Approximation of birth-death processes]{Approximation of birth-death processes}

\author{ Liping Li}
\address{Fudan University, Shanghai, China.  }
\email{liliping@fudan.edu.cn}
\thanks{The author is a member of LMNS,  Fudan University.  He is partially supported by NSFC (No. 11931004
and 12371144). }




\date{\today}




\begin{abstract}
The birth-death process is a special type of continuous-time Markov chain with index set $\mathbb{N}$. Its resolvent matrix can be fully characterized by a set of parameters $(\gamma, \beta, \nu)$, where $\gamma$ and $\beta$ are non-negative constants, and $\nu$ is a positive measure on $\mathbb{N}$. By employing the Ray-Knight compactification, the birth-death process can be realized as a c\`adl\`ag process with strong Markov property on the one-point compactification space $\overline{\mathbb{N}}_{\partial}$, which includes an additional cemetery point $\partial$.  In a certain sense, the three parameters that determine the birth-death process correspond to its killing, reflecting, and jumping behaviors at $\infty$ used for the one-point compactification,  respectively.

In general, providing a clear description of the trajectories of a birth-death process, especially in the pathological case where $|\nu|=\infty$, is challenging. This paper aims to address this issue by studying the birth-death process using approximation methods. Specifically, we will approximate the birth-death process with simpler birth-death processes that are easier to comprehend. For two typical approximation methods, our main results establish the weak convergence of a sequence of probability measures,  which are induced by the approximating processes, on the space of all c\`adl\`ag functions. This type of convergence is significantly stronger than the convergence of transition matrices typically considered in the theory of continuous-time Markov chains.

\end{abstract}

\maketitle

\tableofcontents

\section{Introduction}\label{SEC1}

The birth-death process is a specific type of continuous-time Markov chain with index set $\mathbb{N}$.  Its $Q$-matrix is given by \eqref{eq:11}, and its key characteristic is that it can only transition between adjacent states in $\bN$. Building upon the work of Feller \cite{F59},  Yang demonstrated that all birth-death processes can be obtained by solving the resolvent matrix  (see \cite[Chapter 7]{WY92}).  In essence, each birth-death process is determined by a set of parameters $(\gamma, \beta, \nu)$ that satisfy certain conditions, where $\gamma$ and $\beta$ are two non-negative constants, and $\nu$ is a positive measure on $\mathbb{N}$. For further details, please refer to  \S\ref{SEC23}.

In the study of continuous-time Markov chains, the index set $\mathbb{N}$ is typically equipped with the discrete topology and considered as the state space of the corresponding process. This perspective is reasonable when our focus is limited to the transition matrix or objects related to the distribution of the process.  However, when considering the trajectory of the process, specifically examining its measurability and regularity with respect to time $t$, setting the state space as the discrete space $\mathbb{N}$ can lead to certain ``peculiar" phenomena. The most extreme example is the famous \emph{Feller-McKean chain} (see \cite{FM56}), in which every point in $\mathbb{N}$ is an \emph{instantaneous} state, meaning that the process does not stay at any point for a period of time. Notably, all the diagonal elements of its $Q$-matrix are infinity. In the context of the discrete topology, it is challenging to comprehend the trajectory of the Feller-McKean chain.   However, in reality, the Feller-McKean chain can be derived by confining a regular diffusion process on $\mathbb{R}$ to the set of rational numbers $\mathbb{Q}$ (see \cite[III\S23]{RW00}). Regular diffusions, on the other hand, have a well-established theoretical characterization (see, for example, \cite{IM74}).   Consequently, it is not the Feller-McKean chain itself that is difficult to comprehend, but rather the difficulty arises solely from the setting of the discrete topology.

In the early stages of the development of continuous-time Markov chain theory, Doob realized that the index set $\bN$ is not sufficient to accommodate well-behaved process realizations (see \cite{D53}). He then found a so-called \emph{separable modification} for each continuous-time Markov chain, with trajectories of this modification possessing \emph{Borel measurability}, \emph{separability}, and \emph{lower semicontinuity} (see \cite[II\S4]{C67}). It should be noted that $\infty$ may need to be added to the state space as the compactification point of $\bN$, becoming a state that process trajectories frequently visit (when the index set is $\bZ$, apart from $\infty$, $-\infty$ may also need to be added to the state space). Doob's modification became a cornerstone of continuous-time Markov chain theory and has been widely applied in various aspects of related theory (see \cite{C67} and \cite{WY92}). However, for some special examples, the one-point compactification topology of $\bN$ does not reveal the underlying structural essence of the model. In the case of the Feller-McKean chain, the one-point compactification topology of $\bN$ is far from the induced topology of the Euclidean metric on $\mathbb{Q}$. Doob's modification also falls short of capturing the regular diffusion that generates the Feller-McKean chain. 

Another approach, proposed by Ray in 1959 (see \cite{R59}), can achieve this goal. It is now well known as the \emph{Ray-Knight compactification}. Ray's theory is applicable to almost all processes satisfying the Markov property. It uses the resolvent to introduce a new metric, expanding the state space in a way that completes the metric, and constructs a c\`adl\`ag process satisfying the strong Markov property on the new state space, called a \emph{Ray process}. Ray's approach can also be applied profoundly in the study of continuous-time Markov chains (see \cite[Chapter 6]{RW00} and \cite[Chapter 9]{CW05}), and in some aspects, it is even superior to using Doob's separable modification. Specifically, when Ray-Knight compactification is applied to the Feller-McKean chain, it yields the regular diffusion used to construct the Feller-McKean chain (see \cite[III~(35.7)]{RW00}).

In a previous article \cite{L23}, we investigated all birth-death processes using the Ray-Knight compactification approach. The key advantage of this approach is that it enables us to realize every birth-death process as a c\`adl\`ag process that satisfies the strong Markov property on $\overline{\mathbb{N}}_{\partial}$ (see \S\ref{SEC21} for this symbol). Interestingly, for birth-death processes, both the Doob's modification and the Ray-Knight method result in the same topological transformation. Additionally, with the exception of relatively simple \emph{Doob processes} (see,  e.g.,  \cite[\S5]{L23}), all birth-death processes are Feller processes. Consequently, to study birth-death processes, we have expanded our toolbox beyond traditional methods for studying continuous-time Markov chains to incorporate the rich theory of general Markov processes.

The primary objective of the paper \cite{L23} is to provide a clear characterization of the trajectories of all birth-death processes, particularly their behavior near the boundary point $\infty$. This issue has not been well-addressed in the existing literature based on continuous-time Markov chain methods, such as \cite{WY92, C04}, and others.

Using the framework of Feller processes, we demonstrated in \cite{L23} that the parameters $(\gamma, \beta, \nu)$ determining a birth-death process reflect its different behaviors at the boundary point $\infty$. From an analytical perspective, these behaviors are described by the boundary condition \eqref{eq:21-2} satisfied by the functions in the domain of the infinitesimal generator. From a probabilistic standpoint, $\gamma$, $\beta$, and $\nu$ respectively describe the \emph{killing}, \emph{reflecting}, and \emph{jumping} behaviors of the birth-death process at $\infty$. This probabilistic interpretation is clear when the jumping measure $\nu$ is finite.  More precisely,  Doob process corresponds to the case of $\beta=0$ and $|\nu|<\infty$, where $|\nu|$ is the total variation of $\nu$. It can be  obtained by the \emph{piecing out transformation}  (see \cite{INW66}) of the minimal birth-death process $X^\text{min}$ with respect to the distribution $\pi$ on $\bN\cup \{\partial\}$ given by 
\begin{equation}\label{eq:28}
\pi(\{k\})=\frac{\nu_k}{\gamma+|\nu|},\; k\in \bN,\quad \pi(\{\partial\})=\frac{\gamma}{\gamma+|\nu|}.
\end{equation}
Intuitively, whenever it is about to reach $\infty$, the Doob process always jumps back to $\bN\cup\{\partial\}$, and the probability of arriving at position $k\in \bN\cup\{\partial\}$ is determined by \eqref{eq:28}. For the case of $\beta>0$ and $|\nu|<\infty$, the relevant description requires adjusting the above formulation by considering the birth-death process $X^1$ corresponding to the parameters $(\gamma+|\nu|,\beta,0)$ instead of $X^\text{min}$, and by adapting the random time of approaching $\infty$ to the \emph{lifetime} of $X^1$, namely the time when $X^1$ enters the cemetery $\partial$. Note that $X^1$ can be obtained as a subprocess of the $(Q,1)$-process (which plays a similar role to a reflecting Brownian motion on $[0,\infty)$; see \cite[\S3.3]{L23}) under the \emph{killing transformation} using the \emph{multiplicative functional} 
\[
	M_t:=e^{-\frac{|\nu|+\gamma}{\beta}L_t},\quad t\geq 0,
\]
where $(L_t)_{t\geq 0}$ is the local time of $(Q,1)$-process at $\infty$. For the rigorous construction of the subprocess, readers can refer to \cite[III, \S3]{BG68}.

However, in the case of $|\nu|=\infty$, the construction of piecing out is no longer effective because $\pi$ given by \eqref{eq:28} becomes meaningless. Similar to L\'evy processes with infinite L\'evy measures, in this case, the birth-death process may experience a high frequency of jumps into $\mathbb{N}\cup \{\partial\}$ (from $\infty$) at certain times $t$.  Namely, for any $\varepsilon>0$, there are infinitely many jumps from $\infty$ to $\mathbb{N}\cup\{\partial\}$ occurring within the time interval $[t,t+\varepsilon]$. This makes it extremely difficult to provide a clear description of the trajectories of the birth-death process. Feller referred to this situation as a
``pathological case" in \cite{F59}, probably for this reason.

In this article, we will explore how to use simpler birth-death processes to approximate complex birth-death processes. The significance of this investigation lies in the fact that if we can establish the convergence of the sequence of approximating processes, then even in the pathological case where $\pi$ determined by \eqref{eq:28} lacks meaning, interpreting the target birth-death process through piecing out becomes intuitively acceptable.


Let us first discuss how to obtain simplified birth-death processes for approximation. One approach is to optimize the parameters $(\gamma, \beta, \nu)$ of the target birth-death process and then to apply Feller-Yang's resolvent approach to generate processes using the new parameters. 
The simplest example is to truncate the measure $\nu$ directly as follows:
\begin{equation}\label{eq:13}
\gamma^{(n)}:=\gamma, \quad \beta^{(n)}:=\beta,\quad \nu^{(n)}:=\nu|_{\{0,1,\cdots, n\}},
\end{equation}
where $\nu^{(n)}$ represents the measure $\nu$ restricted to $\{0,1,\cdots, n\}$,  and consider $X^{(n)}$ as the birth-death process determined by the parameters $(\gamma^{(n)}, \beta^{(n)}, \nu^{(n)})$. 
 Note that this approach has been briefly mentioned in \cite[\S9]{L23}.  Another approach, proposed by Wang in his 1958 doctoral thesis (see \cite{WY92}), differs significantly from the first approach but is important in the theory of birth-death processes. Wang constructed a sequence of Doob processes with instantaneous measures $\pi^{(n)}$ supported on $\{0, 1,\cdots, n\}$ by removing the part of each trajectory of the target birth-death process that starts from $\infty$ until it returns to $\{0, 1, \cdots, n\}$.  (The original purpose of this approximation method was to provide a probabilistic construction for all birth-death processes.) We will further elaborate on this approximation method in \S\ref{SEC61}.
Next, let us consider in what sense the constructed sequence of birth-death processes can converge to the target process. Wang's research is based on the theory of continuous-time Markov chains, where the core object is the transition matrix that defines the birth-death process. Therefore, the established convergence also refers to the convergence of the transition matrices, i.e.,
\begin{equation}\label{eq:12}
\lim_{n\rightarrow \infty}p^{(n)}_{ij}(t)=p_{ij}(t),\quad \forall i,j\in \mathbb{N}.
\end{equation}
This convergence is equivalent to the convergence of the resolvents; see Theorem~\ref{THM42}. In the context of general Markov processes, it is possible to extend our study. Since the trajectories of $X^{(n)}$ and $X$ are c\`adl\`ag, they can all be realized as probability measures on the space $D_{\overline{\mathbb{N}}_\partial}[0,\infty)$, which consists of all c\`adl\`ag functions on $\overline{\bN}_\partial$. This space is typically equipped with the Skorohod topology. Consequently, we can attempt to establish the weak convergence of this sequence of probability measures on $D_{\overline{\mathbb{N}}_\partial}[0,\infty)$. This weak convergence is significantly stronger than the convergence of the transition matrices \eqref{eq:12}.

The main goal of this paper is to establish the weak convergence of the probability measures on $D_{\overline{\bN}_\partial}[0,\infty)$ associated with both the optimization parameter approximation and Wang's approximation.  For the first type of approximation,  the weak convergence has been easily established under additional assumptions of Feller properties in \cite[\S9]{L23}. However, when the approximating sequence consists of Doob processes, the discussion becomes more challenging.  Our proof requires an analytical characterization of Doob processes, which, although not Feller processes, can still yield a strongly continuous contractive semigroup when restricted to a closed subspace of the space of continuous functions. This result will be proven in \S\ref{SEC5}. As for Wang's approximation, in addition to establishing weak convergence on the Skorohod topological space, we will also consider another topology on the space $D_{\overline{\bN}_\partial}[0,\infty)$ determined by convergence in (Lebesgue) measure. Under this new topology, we can prove both weak convergence and almost sure convergence of $D_{\overline{\bN}_\partial}[0,\infty)$-valued random variables induced by the birth-death process sequence. In other words, in the topology determined by convergence in  measure, Wang's approximation not only satisfies weak convergence of the  probability measure sequence, but also provides an intuitive construction of the corresponding \emph{Skorohod representation}.

Finally, we would like to briefly explain the notation that will be frequently used in this paper. Since it is a follow-up study of \cite{L23}, we will strive to maintain consistency with the notation used in that paper. However, for the sake of clarity, some symbols have been adjusted. For instance, all symbols related to the minimal birth-death process are denoted with a superscript ${}^\text{min}$, and those related to approximating birth-death processes are represented by superscript ${}^{(n)}$. (Symbols related to the target birth-death process do not carry any superscripts.)  Sometimes, the integral of a function $f$ with respect to a measure $m$ is denoted by $m(f)$.  Additionally, we need to correct an error in the derivation of boundary condition \eqref{eq:21-2} in \cite{L23}. In this equation, the parameter $\frac{\beta}{2}$ in front of $F^+$ was incorrectly written as $\beta$ in \cite{L23}. This error occurred because the scale function used in \cite{L23}, which is given by \eqref{eq:25-2}, is half of the scale function defined in Feller \cite{F59}. However, in the proof of \cite[Theorem~6.3]{L23} (the equation above \cite[(6.8)]{L23}), it was mistakenly overlooked that the result from \cite{F59} (i.e., \eqref{eq:66} in this paper) needs to be multiplied by a factor of two.

\section{Preliminaries of birth-death processes}\label{SEC2}

We consider a birth-death density matrix as follows:
\begin{equation}\label{eq:11}
Q=(q_{ij})_{i,j\in \bN}:=\left(\begin{array}{ccccc}
-q_0 &  b_0 & 0 & 0 &\cdots \\
a_1 & -q_1 & b_1 & 0 & \cdots \\
0 & a_2 & -q_2 & b_2 & \cdots \\
\cdots & \cdots &\cdots &\cdots &\cdots 
\end{array}  \right),
\end{equation}
where $a_k>0$ for $k \geq 1$ and $b_k>0,  q_k=a_k+b_k$ for $k\geq 0$. (Set $a_0=0$ for convenience.) A continuous-time Markov chain $X$ is called a \emph{birth-death $Q$-process} (or simply a \emph{$Q$-process}) if its transition matrix $(p_{ij}(t))_{i,j\in \bN}$ is \emph{standard}, and its \emph{density matrix} is $Q$,  i.e., $p'_{ij}(0)=q_{ij}$ for $i,j\in \bN$.  A $Q$-process is called \emph{honest} if its transition matrix $(p_{ij}(t))_{i,j\in \bN}$ satisfies $\sum_{j\in \bN}p_{ij}(t)=1$ for all $i\in \bN$ and $t\geq 0$.  In our context, two $Q$-processes with the same transition matrix will not be distinguished. For convenience, we will also refer to $(p_{ij}(t))_{i,j\in \bN}$ as a $Q$-process when no confusion arises.  For further terminology concerning continuous-time Markov chains, readers are referred to \cite{C67, WY92}; see also \cite{L23}. 

\subsection{State space}\label{SEC21}

The index set $\bN$ of the transition matrix is typically referred to as the \emph{minimal state space}.  The ``real" state space of a $Q$-process is its one-point compactification (see,  e.g.,  \cite[\S2.1]{L23})
\[
	\overline{\bN}:=\bN \cup \{\infty\},
\]
where $\overline{\mathbb{N}}$ can be metrized with the metric
\[
	r(n,m) = \left| \frac{1}{n+1} - \frac{1}{m+1} \right|,\quad r(n,\infty) = \frac{1}{n+1},\quad n,m\in \bN.
\]
This establishes a topological homeomorphism between $\overline{\mathbb{N}}$ and the set 
\begin{equation}\label{eq:22-2}
\left\{1, \frac{1}{2}, \frac{1}{3}, \dots, \frac{1}{n+1}, \dots, 0\right\},
\end{equation}
equipped with the relative topology of $\bR$.  
Since we do not always consider honest $Q$-processes, it is necessary to introduce a \emph{cemetery point} $\partial$, which lies outside the state space $\overline{\mathbb{N}}$. It should be emphasized that, unless explicitly stated otherwise, $\partial$ is always treated as an isolated point distinct from $\overline{\mathbb{N}}$. For instance, we can define the distance between a state $n\in \overline{\bN}$ and the cemetery point as $r(n,\partial) = \left| \frac{1}{n+1} + 1 \right|$ ($\frac{1}{\infty}:=0$),  where the inclusion of $\partial$ into $\overline{\mathbb{N}}$ is equivalent to adding the point $-1$ to \eqref{eq:22-2}. 

Given the metric $r$,  the set
\[
	\overline{\bN}_\partial :=\overline{\bN}\cup \{\partial\},
\]
forms a compact,  separable  metric space,  and its subspace
\[
	\bN_\partial:=\bN \cup \{\partial\}
\]
is a locally compact,  separable metric space.  For every bounded function $f$ defined on  either $\overline{\bN}_\partial$ or $\bN_\partial$,  we define
\[
	\|f\|_\infty:=\sup_{x\in \overline{\bN}_\partial\text{ or }\bN_\partial}|f(x)|. 
\] 

Let $C(\overline{\bN}_\partial)$ denote the family of all continuous functions on $\overline{\bN}_\partial$,  where we emphasize that for $f\in C(\overline{\bN}_\partial)$,  the value of $f(\partial)$ may no be equal to $0$.  Define
\[
	C(\overline{\bN}):=\left\{f\in C(\overline{\bN}_\partial): f(\partial)=0\right\}.  
\]
The families $C_b(\bN_\partial)$, $C_0(\bN_\partial)$, and $C_c(\bN_\partial)$ consist of all continuous functions on $\bN_\partial$ that are bounded, that vanish at infinity, and that have compact support, respectively.  In particular,  a function $f$ defined on $\bN_\partial$ belongs to $C_0(\bN_\partial)$ (resp.  $C_c(\bN_\partial)$) if $\lim_{n\rightarrow \infty}f(n)=0$ (resp.  if there exists an integer $N$ such that $f(n)=0$ for all $n\geq N$).   Further, we define $C_0(\bN):=\{f\in C_0(\bN_\partial): f(\partial)=0\}$ and $C_c(\bN):=\{f\in C_c(\bN_\partial): f(\partial)=0\}$.  

\subsection{Minimal $Q$-process}

Given the density matrix \eqref{eq:11}, there exists a particular birth-death process known as the \emph{minimal $Q$-process},  which is denoted by $X^\text{min}=\left(X^\text{min}_t \right)_{t\geq 0}$ and whose transition matrix is denoted by $\left(p^\text{min}_{ij}(t)\right)_{i,j\in \bN}$. Analytically, this process corresponds to the minimal solution to the \emph{Kolmogorov backward equation} (as discussed in \cite[\S10]{F59}). From a probabilistic perspective,  it represents a $Q$-process with minimal information, where the trajectories are terminated as they approach $\infty$ at the first time.

We introduce further notations for the minimal $Q$-process.  Let $\zeta^\text{min}:=\inf\{t>0:X^\text{min}_t=\partial\}$ denote the lifetime of $X^\text{min}$,  and define
\begin{equation}\label{eq:22}
	u^\text{min}_\alpha(i):=\bE^\text{min}_i e^{-\alpha \zeta^\text{min}},\quad \alpha>0, i\in \bN.  
\end{equation}
The resolvent matrix of this process is given by
\begin{equation}\label{eq:21}
	\Phi^\text{min}_{ij}(\alpha):=\int_0^\infty e^{-\alpha t}p_{ij}^\text{min}(t)dt,\quad i,j\in \bN,\alpha>0.
\end{equation}
Then,  it is straightforward to verify the following relation:
\begin{equation}\label{eq:24}
	1-u_\alpha^\text{min}(i)=\alpha \sum_{j\in \bN}\Phi^\text{min}_{ij}(\alpha),\quad i\in \bN, \alpha>0.
\end{equation}

Similar to regular diffusions on an interval, the minimal $Q$-process can be fully characterized by two parameters on $\bN$ derived from the matrix \eqref{eq:11}: a \emph{scale function} and a \emph{speed measure}.  The scale function $(c_k)_{k\in \bN}$ is given by
\begin{equation}\label{eq:25-2}
	c_0=0,  \quad c_1=\frac{1}{2b_0}, \quad c_k=\frac{1}{2b_0}+\sum_{i=2}^k \frac{a_1a_2\cdots a_{i-1}}{2b_0b_1\cdots b_{i-1}}, \;k\geq 2,
\end{equation}
 the \emph{speed measure} $\mu$ is
 \[
\mu(\{0\}):=\mu_0=1,  \quad \mu(\{k\}):=\mu_k=\frac{b_0b_1\cdots b_{k-1}}{a_1a_2\cdots a_k}, \; k\geq 1.  
\]
The process $X^\text{min}$ is \emph{symmetric} with respect to $\mu$ in the sense that $\mu_ip_{ij}^\text{min}(t)=\mu_j p_{ji}^\text{min}(t)$ for any $i,j\in \bN$ and $t\geq 0$.  Thus,  the speed measure $\mu$ is also known as the \emph{symmetric measure} of $X^\text{min}$.  Further details and related results are referred to \cite[\S3.1]{L23}, which provides a characterization involving time change transformation of Brownian motion.

Using another two parameters derived from the scale function and the speed measure, we can classify the boundary point $\infty$ in accordance with Feller's approach.  Specifically, we define the following two quantities:
\[
	R:=\sum_{k=0}^\infty (c_{k+1}-c_k)\cdot \sum_{i=0}^k \mu_i, \quad S:=\sum_{k= 0}^\infty c_k \mu_k. 
\]
The following classification for the boundary point $\infty$ is very well known.

\begin{definition}\label{DEF23}
The boundary point $\infty$ (for $X^\text{min}$) is called
\begin{itemize}
\item[(1)] \emph{regular},  if ${R}<\infty,  {S}<\infty$;
\item[(2)] an \emph{exit},  if ${R}<\infty,  {S}=\infty$;
\item[(3)] an \emph{entrance},  if ${R}=\infty,  {S}<\infty$;
\item[(4)] \emph{natural},  if ${R}= {S}=\infty$.
\end{itemize}
\end{definition}
\begin{remark}
Note that $\infty$ is regular if and only if $c_\infty+\mu(\bN)<\infty$,  where
\[
	c_\infty:=\lim_{k\rightarrow \infty} c_k.
\]  
 If $\infty$ is an exit,  then $c_\infty<\infty$ and $\mu(\bN)=\infty$.  If $\infty$ is an entrance,  then $c_\infty=\infty$ and $\mu(\bN)<\infty$.  If $\infty$ is natural,  then $c_\infty+\mu(\bN)=\infty$. 
\end{remark}

It is crucial to highlight that $Q$-processes are unique if and only if $\infty$ is classified as an entrance or natural boundary; for more details, refer to, e.g., \cite[Theorem~3.5]{L23}.  In this paper,  however, we focus on the non-uniqueness case, where $\infty$ is either regular or an exit.  This non-uniqueness implies the condition $c_\infty<\infty$.  Furthermore, it holds that
\begin{equation}\label{eq:23}
	\lim_{i\rightarrow \infty} u^\text{min}_\alpha(i)=1;
\end{equation}
see,  e.g.,  \cite[Lemma~4.1]{L23}.  

\subsection{Parameters determining birth-death processes}\label{SEC23}

From now on, we will assume that $\infty$ is either regular or  an exit.  In addition to the minimal one, consideration of other $Q$-processes, such as Doob processes and the $(Q,1)$-process (which is only applicable in the regular case), is warranted.

 It was first examined by Feller in \cite{F59} and then firmly established by Yang in 1965 (see \cite[Chapter 7]{WY92}) that each (non-minimal) $Q$-process can be uniquely  determined,  up to a multiplicative constant,  by a triple of parameters $(\gamma, \beta,\nu)$.  Here, $\gamma,\beta\geq 0$ are two non-negative constants and $\nu=(\nu_k)_{k\in \bN}$ is a positive measure on $\bN$  satisfying the conditions:
 \begin{equation}\label{eq:B1}
	\sum_{k\geq 0} \nu_k \left(\sum_{j=k}^\infty (c_{j+1}-c_j) \sum_{i=0}^j \mu_i \right)<\infty,\quad |\nu| +\beta\neq 0, 
\end{equation}
where $|\nu|:=\sum_{k\in \bN}\nu_k$,  and
\begin{equation}\label{eq:B4}
\beta=0,\quad \text{if }\infty\text{ is an exit}.
\end{equation}
Let $\mathscr Q$ denote the set of all triples $(\gamma, \beta,\nu)$ satisfying \eqref{eq:B1} and \eqref{eq:B4}.  More precisely,  the resolvent matrix 
\begin{equation}\label{eq:28-2}
	\Phi_{ij}(\alpha):=\int_0^\infty e^{-\alpha t}p_{ij}(t)dt,\quad i,j\in \bN, \alpha>0
\end{equation}
of the $Q$-process $(p_{ij}(t))_{i,j\in \bN}$ corresponding to $(\gamma,\beta,\nu)\in \mathscr Q$ can be expressed as (see,  e.g.,  \cite[Theorem~B.1]{L23})
\begin{equation}\label{eq:B2}
	\Phi_{ij}(\alpha):=\Phi^\text{min}_{ij}(\alpha)+u^\text{min}_\alpha(i) \frac{\sum_{k\in \bN} \nu_k \Phi^\text{min}_{kj}(\alpha)+\beta\mu_j u^\text{min}_\alpha(j)}{\gamma+\sum_{k\in \bN}\nu_k(1-u^\text{min}_\alpha(k))+\beta\alpha\sum_{k\in \bN}\mu_k u^\text{min}_\alpha(k)},
\end{equation}
where $(\Phi^\text{min}_{ij})$ is the minimal resolvent matrix \eqref{eq:21} and $u^\text{min}_\alpha$ is defined as \eqref{eq:22}.  The matrix \eqref{eq:B2} is referred to as the \emph{$(Q, \gamma, \beta, \nu)$-resolvent matrix}.   It should be noted that for any constant $M>0$,  the $(Q,M\gamma, M\beta,  M\nu)$-resolvent matrix is identical to the $(Q,  \gamma,\beta,  \nu)$-resolvent matrix.  
 
We need to provide some observations regarding the expression of the resolvent matrix presented in  \eqref{eq:B2}.  First,  the first inequality in \eqref{eq:B1} is equivalent to
\begin{equation}\label{eq:210}
	\alpha\sum_{k,j\in \bN} \nu_k \Phi^\text{min}_{kj}(\alpha)=\sum_{k\in \bN}\nu_k(1-u^\text{min}_\alpha(k))<\infty,\quad \forall \alpha>0\; (\text{equivalently, } \exists \alpha>0); 
\end{equation}
see,  e.g., \cite[\S7.10]{WY92}.  Second,  $\mu(u^\text{min}_\alpha):=\sum_{k\in \bN}\mu_k u^\text{min}_\alpha(k)$ is finite,  if and only if $\infty$ is regular or an entrance; see, e.g.,  \cite[Theorem~7.1]{F59}.  Thus,  the condition \eqref{eq:B4} guarantees that the resolvent matrix \eqref{eq:B2} is well-defined in the exit case.  Third,  when $\gamma>0$ while $\beta=|\nu|=0$,  the resolvent matrix \eqref{eq:B2} reduces to the minimal one (although the second inequality in \eqref{eq:B1} is not satisfied). 

\subsection{Doob processes and Feller $Q$-processes}

In a recent article \cite{L23}, it was demonstrated, using Ray-Knight compactification, that every $Q$-process $X$ possesses a c\`adl\`ag modification $\bar X$ on $\overline{\bN}_\partial$. This modified version is a \emph{Ray process} on $\overline{\bN}_\partial$. In this paper, we do not differentiate between $X$ and its Ray-Knight compactification $\bar X$. Additionally, \cite[Corollary~5.2]{L23} classifies all non-minimal $Q$-processes into two categories: Doob processes and Feller $Q$-processes.

According to, e.g., \cite[Theorem~B.1]{L23},  the $Q$-process  $X$ is a Doob process,  if and only if its determining triple $(\gamma,\beta,\nu)$ belongs to
\[
	\mathscr Q_D:=\{(\gamma,\beta,\nu)\in \mathscr Q: \beta=0,0<|\nu|<\infty\}.  
\]
In this case,  whenever it approaches $\infty$,  the process $X$ refreshes at a randomly determined location from a distribution $\pi$ given by \eqref{eq:28} on $\mathbb{N}_\partial$.
In other words,  $X$ can be obtained by the \emph{piecing out} (see \cite{INW66}) of the minimal $Q$-process $X^\text{min}$ with respect to the \emph{instantaneous distribution} $\pi$,  as established in \cite[\S5]{L23}.  The readers are also referred to \cite[Appendix~A]{L23} for  a detailed description of the piecing out transformation.  

When $(\gamma,\beta,\nu)\in \mathscr Q_F:=\mathscr Q\setminus \mathscr Q_D$,  the corresponding $Q$-process $X$ is a Feller process on $\overline{\bN}$ (with $\partial$ being the cemetery point).  Following \cite{L23}, this $Q$-process will be termed a \emph{Feller $Q$-process}.  The \emph{infinitesimal generator} of the Feller $Q$-process $X$ is derived in \cite[Theorem~6.3]{L23},  and the crucial point is the boundary condition at $\infty$ satisfied by the functions $F$ in the generator domain: 
\begin{equation}\label{eq:21-2}
\frac{\beta}{2} F^+(\infty)+\sum_{k\in \bN}(F(\infty)-F(k))\nu_k +\gamma F(\infty)=0,
\end{equation}
where $F^+$ represents the discrete gradient of $F$ with respect to the scale function $(c_k)_{k\in \bN}$; see \cite[(6.2)]{L23}.  Based on this boundary condition,  the parameters $(\gamma,\beta, \nu)$ can be utilized to explain three possible types of boundary behaviours of $X$ at $\infty$: \emph{killing},  \emph{reflecting}, and \emph{jumping}.  Please refer to \cite[\S2.3]{L23} and \cite{L24} for more details.   


\subsection{Notations in the context of general Markov processes}

Let us introduce some notations for a $Q$-process in the context of general Markov processes.  

Given the transition matrix $(p_{ij}(t))_{t\geq 0}$ and its resolvent matrix $(\Phi_{ij}(\alpha))_{\alpha>0}$ as provided in \eqref{eq:28-2},  we define 
\[
	R_\alpha f(i):=\sum_{j\in \bN}\Phi_{ij}(\alpha)f(i),\; i\in \bN,\quad R_\alpha f(\partial):=0,
\]
for every $f\in C(\overline{\bN})$ (with $f(\partial)=0$) and $\alpha>0$.  According to \eqref{eq:B2},  $R_\alpha f(i)$ for $f\in C(\overline{\bN})$ and $i\in \bN$ can be expressed in terms of the parameters $(\gamma, \beta,\nu)$ as
\begin{equation}\label{eq:215}
	R_\alpha f(i)=R^\text{min}_\alpha f(i)+u^\text{min}_\alpha(i) \frac{\sum_{k\in \bN} \nu_k R^\text{min}_\alpha f(k)+\beta\sum_{j\in \bN}\mu_j u^\text{min}_\alpha(j)f(j)}{\gamma+\sum_{k\in \bN}\nu_k(1-u^\text{min}_\alpha(k))+\beta\alpha\sum_{k\in \bN}\mu_k u^\text{min}_\alpha(k)},
\end{equation}
where $R^\text{min}_\alpha f(i):=\sum_{j\in \bN}\Phi^\text{min}_{ij}(\alpha)f(i)$.  By utilizing \eqref{eq:24} and \eqref{eq:23},  we can obtain that
\begin{equation}\label{eq:214}
	\lim_{i\rightarrow \infty}R_\alpha f(i)=\frac{\sum_{k\in \bN} \nu_k R^\text{min}_\alpha f(k)+\beta\sum_{j\in \bN}\mu_j u^\text{min}_\alpha(j)f(j)}{\gamma+\sum_{k\in \bN}\nu_k(1-u^\text{min}_\alpha(k))+\beta\alpha\sum_{k\in \bN}\mu_k u^\text{min}_\alpha(k)}.
\end{equation}
Define $R_\alpha f(\infty)$ as this limit.  For $f\in C(\overline{\bN}_\partial)$ where $f(\partial)$ may not be equal to $0$,  it holds that $f_0:=f-f(\partial)\in C(\overline{\bN})$.  We define
\[
	R_\alpha f(x):=R_\alpha f_0(x)+\frac{f(\partial)}{\alpha},\quad f\in C(\overline{\bN}_\partial), x\in \overline{\bN}_\partial.
\]
As established in \cite[\S4]{L23},  
\[
 R_\alpha: C(\overline{\bN}_\partial)\rightarrow C(\overline{\bN}_\partial),\quad \alpha>0
\]
 is a \emph{Ray resolvent} in the sense of,  e.g.,  \cite[Definition~8.1]{CW05}.  According to,  e.g.,  \cite[Theorem~8.2]{CW05},  there exists a Borel measurable Markov transition semigroup $(P_t)_{t\geq 0}$ on $\overline{\bN}_\partial$,  having $(R_\alpha)_{\alpha>0}$ as the resolvent,  such that $t\mapsto P_tf(x)$ is right-continuous for all $f\in C(\overline{\bN}_\partial)$ and $x\in \overline{\bN}_\partial$.  This transition semigroup $(P_t)_{t\geq 0}$ is known as a \emph{Ray semigroup}.   Note that for $f\in C(\overline{\bN})$ and $t\geq 0$,  it holds that $P_tf(i)=\sum_{j\in \bN}p_{ij}(t)f(j)$ for $i\in \bN$ and $P_tf(\partial)=0$. 
 
 In the case where $(p_{ij}(t))_{t\geq 0}$ is a Feller $Q$-process,  the semigroup $(P_t)_{t\geq 0}$ is a \emph{Feller semigroup} in the sense that $P_tC(\overline{\bN}_\partial)\subset C(\overline{\bN}_\partial)$ and 
 \begin{equation}\label{eq:222}
 	\lim_{t\rightarrow 0}\|P_tf-f\|_\infty=0,\quad f\in C(\overline{\bN}_\partial).
 \end{equation}
Particularly,  a Feller $Q$-process satisfies the \emph{normal property} on $\overline{\bN}_\partial$,  i.e.,  $P_0(x,\cdot)=\delta_x$ for all $x\in \overline{\bN}_\partial$.  However, in the case where $(p_{ij}(t))_{t\geq 0}$ is a Doob process,  neither \eqref{eq:222} nor the normal property is satisfied.  More precisely,  $P_0(x,\cdot)=\delta_x$ only holds for $x\in \bN_\partial$,  and $\infty$ is a \emph{branching point} of the $Q$-process in the sense of,  e.g.,  \cite[Definition~8.3]{CW05}.  Specifically,  $P_0(\infty, \cdot)=\pi$,  which is given by \eqref{eq:28}; see \cite[Theorem~5.1]{L23}.  

As a Ray process on $\overline{\bN}_\partial$,  the $Q$-process has a.s.  c\`adl\`ag trajectories on $\overline{\bN}_\partial$ according to,  e.g., \cite[Theorem~8.6]{CW05} (or \cite[Theorem~9.13]{S88}).   Therefore, we can define the trajectory space $\Omega$ as the set of all c\`adl\`ag functions $\omega$ from $[0,\infty)$ to $\overline{\bN}_\partial$ such that $\omega(t)=\partial$ for all  $t\geq \zeta(\omega):=\inf\{t\geq 0:\omega(t)=\partial\}$.  We can then define the projection maps
  \[
  X_t: \Omega \rightarrow \overline{\bN}_\partial,  \quad \omega\mapsto \omega(t)
  \]
  for all $t\geq 0$.
  The translation operators $(\theta_t)_{t\geq 0}$ on $\Omega$ are defined by $\theta_t\omega(s):=\omega(t+s)$ for all $t,s\geq 0$.  Let $\sF^0:=\sigma\left\{X_s:s\geq 0\right\}$ and $\sF^0_t:=\sigma\left\{X_s:0\leq s\leq t\right\}$,  the $\sigma$-algebras on $\Omega$ generated by $\{X_s:s\geq 0\}$ and $\{X_s:0\leq s\leq t\}$,  respectively.  These $\sigma$-algebras are known as the \emph{natural filtration} on $\Omega$.  According to \cite[Theorem~8.6]{CW05} (see also \cite[Theorem~9.13]{S88}),  for any probability measure $\lambda$ on $\overline{\bN}_\partial$,  there exists a probability measure $\bP_\lambda$ on $(\Omega, \sF^0)$ such that 
  \[
  	\left(\Omega, \sF^0,\sF^0_t, X_t, \theta_t, \bP_\lambda \right)
  \]
 forms a Markov process on $\overline{\bN}_\partial$ with initial distribution $\lambda P_0$ (not $\lambda$!) and transition semigroup $(P_t)_{t\geq 0}$. Here,  $\lambda P_0(A):=\int_{\overline{\bN}_\partial}P_0(x,A)\lambda(dx)$ for $A\subset \overline{\bN}_\partial$.  If $\lambda=\delta_x$ for $x\in \overline{\bN}_\partial$,  we write $\bP_\lambda$ as $\bP_x$.  Additionally, note that $\bP_\lambda(\cdot)=\int_{\overline{\bN}_\partial} \bP_x(\cdot)\lambda(dx)$. The natural filtration $(\sF^0,\sF^0_t)$ can be augmented using the standard approach described in \cite[I\S6]{S88}, resulting in the \emph{augmented natural filtration} $(\sF,\sF_t)$ on $\Omega$.  Finally,  we  obtain a collection
  \begin{equation}\label{eq:213}
  	X=\left(\Omega, \sF,\sF_t, X_t,\theta_t, \bP_x \right),
  \end{equation}
 which forms a \emph{realization} of the $Q$-process $(p_{ij}(t))_{i,j\in \bN}$.  
  
The lifetime of $X$ is $\zeta=\inf\{t>0:X_t=\partial\}$.  Note that $\bP_\partial(X_t=\partial,\forall t\geq 0)=1$ and $\bP_i(X_0=i)=1$ for all $i\in \bN$.  In the case where $X$ is a Feller $Q$-process,  it also holds that $\bP_\infty(X_0=\infty)=1$.  In contrast,  for a Doob process, $\bP_\infty(X_0\in \cdot)=\pi$ since $P_0(\infty, \cdot)=\pi$.  However,  when the Doob process is restricted to $\bN_\partial$, it transforms into a \emph{Borel right process} that satisfies the normal property (see \cite[Theorem~9.13]{S88}). It is worth noting that in the context of Borel right process, the cemetery point $\partial$ is commonly considered as the compactification point of the state space $\bN$, which slightly differs from the setup in  \S\ref{SEC21}.

 \subsection{Realization on Skorohod topological space}\label{SEC26}

Let $D_{\overline{\bN}_\partial}[0,\infty)$ denote the set of all c\`adl\`ag functions from $[0,\infty)$ to $\overline{\bN}_\partial$.  
According to \cite[\S3, Theorem~5.6]{EK09},   the space $D_{\overline{\bN}_\partial}[0,\infty)$ equipped with the metric $d$ inducing the \emph{Skorohod topology}  (defined in \cite[\S3(5.2)]{EK09})  is a complete,  separable metric space.  In addition,  utilizing \cite[\S3, Proposition~7.1]{EK09},  we can identify the Borel $\sigma$-algebra $\mathscr{B}\left(D_{\overline{\bN}_\partial}[0,\infty);d\right)$ on $D_{\overline{\bN}_\partial}[0,\infty)$ with respect to the Skorohod topology as $\sigma\{\pi_t:t\geq 0\}$,  which is generated by all the projection maps $\pi_t:D_{\overline{\bN}_\partial}[0,\infty)\rightarrow \overline{\bN}_\partial,  w\mapsto w(t)$.  
 
It is evident that $\Omega\subset D_{\overline{\bN}_\partial}[0,\infty)$,   where $\Omega$ is the trajectory space in the realization \eqref{eq:213} of the $Q$-process.  Since $\mathscr{B}\left(D_{\overline{\bN}_\partial}[0,\infty);d\right)=\sigma\{\pi_t:t\geq 0\}$,  it is straightforward to verify that the embedding map
\begin{equation}\label{eq:217}
	\mathcal X: (\Omega, \sF)\rightarrow \left(D_{\overline{\bN}_\partial}[0,\infty), \mathscr{B}\left(D_{\overline{\bN}_\partial}[0,\infty)\right);d \right),\quad \omega\mapsto X_\cdot (\omega)
\end{equation}
is measurable.  Thus,  for any probability measure $\lambda$ on $\overline{\bN}_\partial$,  $\bP_\lambda$ induces an image probability measure $\bP_\lambda \circ \mathcal{X}^{-1}$  on $\left(D_{\overline{\bN}_\partial}[0,\infty), \mathscr{B}\left(D_{\overline{\bN}_\partial}[0,\infty)\right);d \right)$.  Since $\bP_\lambda\circ \mathcal{X}^{-1}$ can be regarded as the extension of $\bP_\lambda$ to $D_{\overline{\bN}_\partial}[0,\infty)$ by defining $\bP_\lambda(D_{\overline{\bN}_\partial}[0,\infty)\setminus \Omega):=0$,  we will still denote this image measure by $\bP_\lambda$ if no ambiguity arises.

 \section{Convergence of resolvents}\label{SEC31}
 
Consider a $Q$-process $X$ with parameters $(\gamma, \beta, \nu) \in \mathscr{Q}$, which determine its resolvent matrix. Define $u_\alpha(k):=\bE_ke^{-\alpha \zeta}$ for $k\in \bN$ and $\alpha$, where $\zeta$ is the lifetime of $X$.   Additionally, consider a sequence of $Q$-processes $\{X^{(n)}: n \geq 1\}$, with parameters denoted by $(\gamma^{(n)}, \beta^{(n)}, \nu^{(n)}) \in \mathscr{Q}$. Symbols related to this sequence will be distinguished by the superscript $\text{}^{(n)}$.  For example, the realization of $X^{(n)}$ can be denoted by
\[
X^{(n)} = \left(\Omega, \sF^{(n)}, \sF^{(n)}_t, X_t, \theta_t, \bP^{(n)}_x \right),
\]
where the lifetime is $\zeta^{(n)}$. The semigroup and resolvent of $X^{(n)}$ are denoted by $(P^{(n)}_t)_{t\geq 0}$ and $(R^{(n)}_\alpha)_{\alpha>0}$, respectively,  and $u^{(n)}_\alpha(k):=\bE^{(n)}_k e^{-\alpha \zeta^{(n)}}$ for $k\in \bN$ and $\alpha>0$.  

In this section,  our aim is to clarify the relationships between the convergence of the transition matrices, the resolvent matrices, the transition semigroups, and the resolvents for $X^{(n)}$.  Among these convergences, the resolvent convergence is comparatively clear and straightforward,  according to the resolvent representation \eqref{eq:215}.

A simple case of Kurtz's lemma \cite[Lemma~2.11]{K69},  as stated below,  will be useful in proving our results.

\begin{lemma}\label{LM91}
Let $\mathbf{B}$ be a Banach space with the norm $\|\cdot\|$.  Suppose that for each $n$,  $F_n(t)$ is a function of $0\leq t<\infty$ taking values in $\mathbf{B}$, and that $F_n(t)$ forms a bounded,  equicontinuous sequence in the sense that:
\begin{itemize}
\item[(i)] There exists $M>0$ such that $\|F_n(t)\|\leq M$ for all $n$ and $t\geq 0$.
\item[(ii)] For every $\varepsilon>0$ and $t\geq 0$,  there exists $\delta>0$ such that $t'\geq 0$ with $|t-t'|<\delta$ implies $\|F_n(t)-F_n(t')\|<\varepsilon$ for all $n$.  
\end{itemize}  
Then
\begin{equation}\label{eq:90}
\lim_{n\rightarrow \infty} \left\| \int_0^\infty e^{-\alpha t}F_n(t)dt\right\|=0\quad \text{for all }\alpha >0
\end{equation}
implies 
\[
	\lim_{n\rightarrow \infty} \sup_{0\leq t\leq T}\|F_n(t)\|=0\quad \text{for all }T>0.   
\]
\end{lemma}

Now we are in a position to present our first result regarding the equivalent conditions for resolvent convergence.

\begin{theorem}\label{THM42}
The following convergences are all equivalent to each other:
\begin{itemize}
\item[(1a)] For some $k\in \bN$ (or equivalently,  for all $k\in \bN$),  it holds that $\lim_{n\rightarrow \infty} R^{(n)}_\alpha f(k)=R_\alpha f(k)$ for all $f\in C(\overline{\bN})$ and $\alpha>0$.
\item[(1b)] $\lim_{n\rightarrow \infty} \|R^{(n)}_\alpha f-R_\alpha f\|_\infty=0$ for all  $f\in C(\overline{\bN})$ and $\alpha>0$. 
\item[(1c)] $\lim_{n\rightarrow \infty} R^{(n)}_\alpha f(\infty)=R_\alpha f(\infty)$ for all  $f\in C(\overline{\bN})$ and $\alpha>0$. 
\item[(2a)] For some $k\in \bN$ (or equivalently,  for all $k\in \bN$),  it holds that $\lim_{n\rightarrow \infty}\Phi^{(n)}_{kj}(\alpha)=\Phi_{kj}(\alpha)$ for all $j\in \bN$ and $\alpha>0$, and $\lim_{n\rightarrow \infty}u^{(n)}_\alpha(k)=u_\alpha(k)$ for all $\alpha>0$.  
\item[(2b)] It holds that $\lim_{n\rightarrow \infty}\sup_{k\in \bN}|\Phi^{(n)}_{kj}(\alpha)-\Phi_{kj}(\alpha)|=0$ for all $j\in \bN$ and $\alpha>0$,  and $\lim_{n\rightarrow \infty}\sup_{k\in \bN}|u^{(n)}_\alpha(k)-u_\alpha(k)|=0$ for all $\alpha>0$.  
\item[(3a)] For some $k\in \bN$ (or equivalently,  for all $k\in \bN$),  it holds that $\lim_{n\rightarrow \infty}P^{(n)}_tf(k)=P_tf(k)$ for all $f\in C_0(\bN)$ and $t\geq 0$, and $\lim_{n\rightarrow \infty}u^{(n)}_\alpha(k)=u_\alpha(k)$ for all $\alpha>0$.
\item[(3b)] For some $k\in \bN$ (or equivalently,  for all $k\in \bN$),  it holds that 
\[
	\lim_{n\rightarrow \infty}\sup_{t\in [0,T]}|P^{(n)}_tf(k)-P_tf(k)|=0
\]
for all $f\in C_0(\bN)$ and $T\geq 0$, and $\lim_{n\rightarrow \infty}u^{(n)}_\alpha(k)=u_\alpha(k)$ for all $\alpha>0$.
\item[(4a)] For some $k\in \bN$ (or equivalently,  for all $k\in \bN$),  it holds that $\lim_{n\rightarrow \infty}p^{(n)}_{kj}(t)=p_{kj}(t)$ for all $j\in \bN$  and $t\geq 0$, and $\lim_{n\rightarrow \infty}u^{(n)}_\alpha(k)=u_\alpha(k)$ for all $\alpha>0$.  
\item[(4b)] For some $k\in \bN$ (or equivalently,  for all $k\in \bN$),  it holds that 
\begin{equation}\label{eq:42}
\lim_{n\rightarrow \infty}\sup_{t\in [0,T]}|p^{(n)}_{kj}(t)-p_{kj}(t)|=0
\end{equation}
for all $j\in \bN$ and $T\geq 0$, and $\lim_{n\rightarrow \infty}u^{(n)}_\alpha(k)=u_\alpha(k)$ for all $\alpha>0$.  
\end{itemize}
\end{theorem}
\begin{proof}
The equivalence between (1a),  (1b),  and (1c) can be easily verified by considering the following facts: $R^{(n)}_\alpha f, R_\alpha f\in C(\overline{\bN})$ for $f\in C(\overline{\bN})$ and $\alpha>0$; according to \eqref{eq:215},
\[
	R^{(n)}_\alpha f(k)-R_\alpha f(k)=u^\text{min}_\alpha(k)\cdot \left(R^{(n)}_\alpha f(\infty)-R_\alpha f(\infty) \right),\quad k\in \bN;
\]
in addition, $0<u_\alpha^\text{min}(k)<1$ for all $k\in \bN$.  

Clearly,  (2b) implies (2a).  By taking $f=1_{\{j\}}$ and $f=1_{\overline{\bN}}$ in (1b),  we obtain 
\[
	\lim_{n\rightarrow \infty}\sup_{k\in \bN}|\Phi^{(n)}_{kj}(\alpha)-\Phi_{kj}(\alpha)|=0\quad\text{and}\quad \lim_{n\rightarrow \infty}\sup_{k\in \bN}|u^{(n)}_\alpha(k)-u_\alpha(k)|=0,
\] 
respectively.  Thus,  (1b) indicates (2b).  Now,  we will demonstrate that (2a) implies (1a), thereby establishing the equivalence between (1a), (1b), (1c), (2a), and (2b).  Suppose that for some $k\in \bN$, it holds that $\lim_{n\rightarrow \infty}\Phi^{(n)}_{kj}(\alpha)=\Phi_{kj}(\alpha)$ for all $j\in \bN$ and $\lim_{n\rightarrow \infty}u^{(n)}_\alpha(k)=u_\alpha(k)$.  For any $g\in C_c(\bN)$,  there exists an integer $M$ such that $g(m)=0$ for all $m>M$.  Therefore, we have
\[
	\lim_{n\rightarrow \infty}R^{(n)}_\alpha g(k)=\lim_{n\rightarrow \infty}\sum_{0\leq m\leq M}\Phi^{(n)}_{km}(\alpha)g(m)=\sum_{0\leq m\leq M}\Phi_{km}(\alpha)g(m)=R_\alpha g(k).
\]
Note that $C_c(\bN)$ is dense in $C_0(\bN)$,  and $\|R^{(n)}_\alpha g\|_\infty\leq \frac{1}{\alpha}\|g\|_\infty,  \|R_\alpha g\|_\infty\leq \frac{1}{\alpha}\|g\|_\infty$ for all $g\in C(\overline{\bN})$.  It is straightforward to further obtain that $\lim_{n\rightarrow \infty}R^{(n)}_\alpha g(k)=R_\alpha g(k)$ for all $g\in C_0(\bN)$.  Taking $f\in C(\overline{\bN})$,  and defining $f_0:=f-f(\infty)\cdot 1_{\overline{\bN}}\in C_0(\bN)$,  we observe that
\[
	R^{(n)}_\alpha 1_{\overline{\bN}}(k)=\frac{1}{\alpha}\left(1-u_\alpha^{(n)}(k)\right),\quad R_\alpha 1_{\overline{\bN}}(k)=\frac{1}{\alpha}\left(1-u_\alpha(k)\right).
\]
From this, we can deduce that
\[
\begin{aligned}
\lim_{n\rightarrow \infty}R^{(n)}_\alpha f(k)&=\lim_{n\rightarrow \infty}\left(R^{(n)}_\alpha f_0(k)+\frac{f(\infty)}{\alpha} \left(1-u_\alpha^{(n)}(k)\right)\right) \\
&=R_\alpha f_0(k)+\frac{f(\infty)}{\alpha} \left(1-u_\alpha(k)\right)=R_\alpha f(k).
\end{aligned}\] 
Consequently,  (1a) holds true.

Next,  we will establish the equivalence between (4a),  (4b), and (2a).  Clearly,  (4b) implies (4a),  and (4a) implies (2a) (by the dominated convergence theorem).  Suppose that (2a) holds.  In order to conclude (4b),  our goal is to apply Lemma~\ref{LM91} with $\mathbf{B}=\mathbb{R}$ and $F_n(t):=p^{(n)}_{kj}(t)-p_{kj}(t)$.  It is sufficient to verify condition (ii) of Lemma~\ref{LM91}.  In fact,   based on \cite[II\S3, Theorem~1]{C67},  for any $t,t'\geq 0$,  we have
\begin{equation}\label{eq:96}
	|F_n(t)-F_n(t')|\leq |p^{(n)}_{kj}(t)-p^{(n)}_{kj}(t')|+|p_{kj}(t)-p_{kj}(t')|\leq 2q_k|t-t'|. 
\end{equation}
Hence, the equcontinuity of $F_n(t)$ can be easily obtained.  

For the remaining conditions (3a) and (3b),  it is worth noting that (3b) implies (3a),  and (3a) implies (4a) (by taking $f=1_{\{j\}}$ in (3a)). Additionally, (4a) implies (4b).  Furthermore,  we can use the same argument as the one used to show that (2a) implies (1a) to deduce that (4b) implies (3b).  Therefore, the equivalence of all conditions has been established.
\end{proof}
\begin{remark}
In discussions about general Markov processes, it is commonly assumed that functions take the value of $0$ at the cemetery state $\partial$. This assumption is the reason why we consider $f \in C(\overline{\bN})$ in the theorem above. However, in the subsequent discussion in \S\ref{SEC4}, we will examine the weak convergence of a family of probability measures on $D_{\overline{\bN}_\partial}[0, \infty)$, where it becomes necessary to address functions that are not equal to $0$ at $\partial$.  Fortunately, it is easy to adjust the conditions in the theorem to apply to $f \in C(\overline{\mathbb{N}}_\partial)$. To keep the presentation simple, we will not elaborate on this here.
\end{remark}

In conditions (3a) and (3b) that describe the convergence of the transition semigroups, it is somewhat limiting that $f$ can only be chosen from functions in $C_0(\bN)$ rather than from all functions in $C(\overline{\bN})$.  Moreover, there is a lack of characterization for equivalent conditions where $k$ has uniformity.  These limitations render it inadequate to guarantee the convergence of the finite-dimensional distributions of $X^{(n)}$ to the corresponding finite-dimensional distributions of $X$.  The challenge here is that if we plan to apply Lemma~\ref{LM91} with $\mathbf{B} = C(\overline{\bN})$ and $F(t) := P^{(n)}_t f - P_t f$, neither $X^{(n)}$ nor $X$ is necessarily a Feller $Q$-process, which does not always guarantee that $F(t) \in \mathbf{B}$. To address this difficulty, we need to take an alternative approach using Theorem~\ref{THM64}.  We will revisit this issue in \S\ref{SEC4}.

 \section{Infinitesimal generators of Doob processes} \label{SEC5}
 
 In this section,  we examine the Doob process, whose resolvent matrix is determined by the parameters \((\gamma, 0, \nu) \in \mathscr{Q}_D\).  All other notation related to $X$ is consistent with that in \S\ref{SEC2}.
 Note that the resolvent $R_\alpha: C(\overline{\bN}) \rightarrow C(\overline{\bN})$ does not satisfy the strong continuity.  That is, $\lim_{\alpha \rightarrow \infty} \|\alpha R_\alpha f - f\|_\infty = 0$ does not hold for all $f \in C(\overline{\bN})$.  However, we will demonstrate that the strong continuity does hold in a smaller Banach space, enabling the transition semigroup $(P_t)_{t\geq 0}$ of the Doob process to form a ``Feller semigroup" on this Banach space.   
 
 We first introduce a lemma,  which states that the minimal $Q$-process is a Feller process on $\bN$ (with $\partial$ being the cemetery).  
 
 \begin{lemma}\label{LM31}
 The transition semigroup $(P^\text{min}_t)_{t\geq 0}$ of $X^\text{min}$ acts on $C_0(\bN)$ as a strongly continuous contractive semigroup,  i.e.,  
 \begin{equation}\label{eq:32}
 	P^\text{min}_t f\in C_0(\bN),\quad \lim_{t\rightarrow 0}\|P_t^\text{min} f-f\|_\infty=0
 \end{equation}
 for all $f\in C_0(\bN)$.  
 \end{lemma}
 \begin{proof}
Let $(\sA^\text{min},\sG^\text{min})$ be the Dirichlet form associated with $X^\text{min}$ on $L^2(\bN,\mu)$,  as described in \cite[Lemma~3.1]{L23}. Fix $f\in C_c(\bN)\subset \sG^\text{min}$.  According to \cite[Lemma~3.1]{L23},  we have $R^\text{min}_\alpha f \in \sG^\text{min}\subset  C_0(\bN)$ for any $\alpha>0$.  By utilizing \cite[Lemma~1.3.3]{FOT11},  we can derive that 
\[
	\sA^\text{min}(g_\alpha,g_\alpha):=\frac{1}{2}\sum_{k\in \bN}\frac{\left(g_\alpha(k+1)-g_\alpha(k)\right)^2}{c_{k+1}-c_k}\rightarrow 0
\]
as $\alpha\rightarrow \infty$,  where $g_\alpha:=\alpha R^\text{min}_\alpha f-f\in C_0(\bN)$.  It follows from the Cauchy-Schwarz inequality that for any $k_0\in \bN$,
\[
|g_\alpha(k_0)|=\left|\sum_{k\geq k_0}(g_{\alpha}(k+1)-g_\alpha(k)) \right| \leq \left(2\sA^\text{min}(g_\alpha,g_\alpha)\right)^{1/2}\cdot c_\infty^{1/2}.
\]
Consequently,  $\lim_{\alpha\rightarrow \infty}\|\alpha R^\text{min}_\alpha f-f\|_\infty=\lim_{\alpha\rightarrow \infty}\|g_\alpha\|_\infty=0$.  In other words,  we have established that $R^\text{min}_\alpha C_c(\bN)\subset C_0(\bN)$ and $\lim_{\alpha\rightarrow \infty}\|\alpha R^\text{min}_\alpha f-f\|_\infty=0$ for all $f\in C_c(\bN)$.  Note that $C_c(\bN)$ is dense in $C_0(\bN)$ with respect to the uniform norm $\|\cdot\|_\infty$,  and $\|\alpha R^\text{min}_\alpha f\|_\infty \leq \|f\|_\infty$ holds for all $f\in C_c(\bN)$.  Therefore,  it is straightforward to verify that $R_\alpha^\text{min}C_0(\bN)\subset C_0(\bN)$ and $\lim_{\alpha\rightarrow \infty}\|\alpha R^\text{min}_\alpha f-f\|_\infty=0$ for all $f\in C_0(\bN)$.   Applying the Hille-Yosida theorem,  we can obtain \eqref{eq:32}.  
 \end{proof}

Based on the density matrix $Q$ provided by \eqref{eq:11},  we define a function $QF$ on $\bN$ for every function $F$ on $\bN$ as
\begin{equation}\label{eq:31}
	QF(k):=a_kF(k-1)-q_kF(k)+b_kF(k+1),\quad k\in \bN,
\end{equation}
where $F(-1):=0$.  Note that $F$ can be extended to a function in $C(\overline{\bN})$,  if and only if $F(\infty):=\lim_{n\rightarrow \infty}F(n)$ exists.  In this case, we represent its extension using the same symbol $F$.  Then,  by an abuse of notation, the operator $Q$ defined by \eqref{eq:31} with domain 
\[
	\cD(Q):=\{F\in C(\overline{\bN}): QF\in C(\overline{\bN})\}
\]
is usually referred to as the \emph{maximal (discrete) generalized second order differential operator};  see,  e.g.,  \cite{F59} and \cite[Lemma~6.1]{L23}. 

Given the triple $(\gamma, 0,\nu)\in \mathscr Q_D$,  we define
\begin{equation}\label{eq:611-3}
	\bC_{\nu,\gamma}:=\left\{F\in C(\overline \bN): \sum_{k\in \bN}(F(\infty)-F(k))\nu_k +\gamma F(\infty)=0\right\}.
\end{equation}
Since $|\nu|<\infty$,  $\bC_{\nu,\gamma}$ is a closed subspace of $C(\overline \bN)$.  Thus,  it is a Banach space equipped with the uniform norm $\|\cdot\|_\infty$.  

The following result is analogous to the second part of \cite[II \S5, Theorem~3]{M68}. 

\begin{theorem}\label{THM64}
Let $X$ be a Doob process determined by the triple $(\nu, 0,\gamma)\in \mathscr Q_D$,  i.e.,  $0<|\nu|<\infty$.   
Then, $(P_t)_{t\geq 0}$ acts on $\bC_{\nu,\gamma}$ as a strongly continuous contractive  semigroup,  i.e., 
\[
	P_t f\in \bC_{\nu,\gamma},\quad \lim_{t\rightarrow 0}\|P_tf-f\|_\infty=0
\]
for all $f\in \bC_{\nu,\gamma}$.  
Furthermore,  the infinitesimal generator of $(P_t)_{t\geq 0}$ on $\bC_{\nu,\gamma}$ is $\sL_{\nu,\gamma}F:=QF$ with domain $\cD(\sL_{\nu,\gamma}):=\{F\in \bC_{\nu,\gamma}: QF\in \bC_{\nu,\gamma}\}$.  
\end{theorem}
\begin{proof}
We aim to apply the Hille-Yosida theorem (see, e.g., \cite[I \S1, Theorem~1]{M68}) to $(R_\alpha)_{\alpha>0}$ and $\sL_{\nu,\gamma}$.  Two facts need to be proven:
\begin{itemize}
\item[(1)] $\cD(\sL_{\nu,\gamma})$ is dense in $\bC_{\nu,\gamma}$.
\item[(2)] For any $f \in \bC_{\nu,\gamma}$ and $\alpha > 0$, $R_\alpha f$ is the unique solution to the equation
\begin{equation}\label{eq:617}
\alpha F - \sL_{\nu,\gamma} F = f.
\end{equation}
\end{itemize}

Firstly,  we prove that $R_\alpha f\in \D(\sL_{\nu,\gamma})$ for any $f\in \bC_{\nu,\gamma}$ and $\alpha>0$.   According  to \eqref{eq:214},  we have $R_\alpha f\in C(\overline{\bN})$ with
\begin{equation}\label{eq:618}
\begin{aligned}
	R_\alpha f(k)&=R^\text{min}_\alpha f(k)+u^\text{min}_\alpha(k) R_\alpha f(\infty)\\
	&=R^\text{min}_\alpha f(k)+u^\text{min}_\alpha(k)\frac{\sum_{i\in \bN}\nu_i R^\text{min}_\alpha f(i)}{\gamma+\sum_{i\in \bN}\nu_i(1-u^\text{min}_\alpha(i))}.
\end{aligned}\end{equation}
Then, a straightforward computation yields
\[
	\sum_{k\in \bN}\nu_k R_\alpha f(k)=(\gamma+|\nu|)R_\alpha f(\infty).
\]
Hence, $R_\alpha f\in \bC_{\nu,\gamma}$.  In addition,  it follows from \cite[Theorems~7.1 and 9.1]{F59} that for $k\in \bN$, 
\begin{equation}\label{eq:620}
\begin{aligned}
	QR_\alpha f(k)&=QR^\text{min}_\alpha f(k)+R_\alpha f(\infty)\cdot Qu^\text{min}_\alpha (k)  \\
	&=(\alpha R^\text{min}_\alpha f(k)-f(k))+R_\alpha f(\infty) \cdot \alpha u^\text{min}_\alpha (k).  
\end{aligned}\end{equation}
Thus, $QR_\alpha f\in C(\overline{\bN})$ with 
\begin{equation}\label{eq:619}
QR_\alpha f(\infty)=\alpha R_\alpha f(\infty)-f(\infty).  
\end{equation}
Note that $\sum_{k\in \bN}\nu_k f(k)=(\gamma+|\nu|)f(\infty)$ due to $f\in \bC_{\nu,\gamma}$.  This,  together with \eqref{eq:618} and \eqref{eq:619},  yields that 
\[
	\sum_{k\in \bN}\nu_k QR_\alpha f(k)=(\gamma+|\nu|)QR_\alpha f(\infty).
\]
As a result, $QR_\alpha f\in \bC_{\nu,\gamma}$.  Therefore,  $R_\alpha f\in \cD(\sL_{\nu,\gamma})$ is obtained.   

Next,  we consider the resolvent equation \eqref{eq:617} and prove that $R_\alpha f$ is the unique solution to \eqref{eq:617}.  In fact,  by using $R_\alpha f\in \cD(\sL_{\nu,\gamma})$, \eqref{eq:618},  and \eqref{eq:620},  we have
\[
	\left((\alpha-\sL_{\nu,\gamma})R_\alpha f\right)(k)=\alpha R_\alpha f(k)-Q R_\alpha f(k)=f(k),\quad \forall k\in \bN.
	\]
Consequently, $R_\alpha f$ is a solution to \eqref{eq:617}.  
The uniqueness of the solutions to \eqref{eq:617} can be concluded by the same argument in the last paragraph of the proof of \cite[Theorem~6.3]{L23}.  

Thirdly,  we demonstrate that 
\begin{equation}\label{eq:38}
	\lim_{\alpha\rightarrow \infty}\|\alpha R_\alpha f-f\|_\infty=0
\end{equation}
for all $f\in \bC_{\nu,\gamma}$,  thereby establishing that $\cD(\sL_{\nu,\gamma})$ is dense in $\bC_{\nu,\gamma}$.  To prove \eqref{eq:38},  we define $f_0:=f-f(\infty)\cdot 1_{\overline{\bN}}\in C_0(\bN)$.  Utilizing \eqref{eq:24} and \eqref{eq:618},  we can deduce that
\begin{equation}\label{eq:39}
\begin{aligned}
	\alpha R_\alpha f(k)-f(k)&=\alpha R^\text{min}_\alpha f_0(k)-f_0(k)-f(\infty)u_\alpha^\text{min}(k)\\
	&\qquad  +u^\text{min}_\alpha(k)\frac{\nu\left(\alpha R^\text{min}_\alpha f_0\right)+f(\infty)\sum_{i\in \bN}\nu_i(1-u^\text{min}_\alpha(i))}{\gamma+\sum_{i\in \bN}\nu_i(1-u^\text{min}_\alpha(i))}.  
\end{aligned}\end{equation}
As established in Lemma~\ref{LM31},  $\lim_{\alpha \rightarrow \infty}\|\alpha R_\alpha^\text{min}f_0-f_0\|_\infty=0$.  Since $|\nu|<\infty$ and $\lim_{\alpha\rightarrow \infty}u^\text{min}_\alpha(i)=0$,  it follows that
\[
\lim_{\alpha\rightarrow \infty}\frac{\nu\left(\alpha R^\text{min}_\alpha f_0\right)+f(\infty)\sum_{i\in \bN}\nu_i(1-u^\text{min}_\alpha(i))}{\gamma+\sum_{i\in \bN}\nu_i(1-u^\text{min}_\alpha(i))}=\frac{\nu(f_0)+f(\infty)|\nu|}{\gamma+|\nu|}.
\]
Note that $(\gamma+|\nu|)f(\infty)=\nu(f)=\nu(f_0)+f(\infty)|\nu|$,  since $f\in \bC_{\nu,\gamma}$.  Thus,  the above limit is equal to $f(\infty)$.  From \eqref{eq:39},  we can obtain \eqref{eq:38}.  

Finally,  by applying the Hille-Yosida theorem,  we can conclude that $\sL_{\nu,\gamma}$ is the infinitesimal generator of the resolvent $(R_\alpha)_{\alpha>0}$ on $\bC_{\nu,\gamma}$.  Hence,  it admits a strongly continuous contractive semigroup $(P'_t)_{t\geq 0}$ on $\bC_{\nu, \gamma}$.  Fixing $f\in \bC_{\nu,\gamma}$ and $x\in \overline\bN$,  we have
\[
	R_\alpha f(x)=\int_0^\infty e^{-\alpha t}P_tf(x)dt=\int_0^\infty e^{-\alpha t}P'_tf(x)dt,\quad \forall \alpha>0.
\]
Thus,  $P_tf(x)=P'_t f(x)$ for a.e.  $t$.  Since $\|P'_{t+h}f-P'_t f\|_\infty\rightarrow 0$ as $h\downarrow 0$ for any $t\geq 0$ and $t\mapsto P_t f(x)$ is right-continuous (see \cite[Theorem~8.2]{CW05}),  it is easy to see that $P_tf(x)=P'_tf(x)$ for all $t$.  This completes the proof. 
\end{proof}

\begin{remark}
In the reduced case where $\gamma>0$ and $\beta=|\nu|=0$, corresponding to the minimal $Q$-process,  we have $\bC_{\nu,\gamma}=C_0(\bN)$.  The same argument as presented in this proof indicates that the infinitesimal generator of $(P^\text{min}_t)_{t\geq 0}$ on $C_0(\bN)$ is given by $\sL^\text{min}F=QF$ with domain $\cD(\sL^\text{min})=\{F\in C_0(\bN):QF\in C_0(\bN)\}$.  
\end{remark}

\section{Weak convergence on Skorohod topological space}\label{SEC4}

In this section, we continue our study of the convergence issues that were temporarily paused in \S\ref{SEC31}.  As explained in \S\ref{SEC26},  for each probability measure $\lambda^{(n)}$ (resp.  $\lambda$) on $\overline{\bN}_\partial$,  $\bP^{(n)}_{\lambda^{(n)}}$ (resp.  $\bP_\lambda$) can be considered as a probability measure on the Skorohod topological space $D_{\overline{\bN}_\partial}[0,\infty)$.  Our goal is to demonstrate that if $(\gamma^{(n)},\beta^{(n)},\nu^{(n)})$ converges to $(\gamma,\beta,\nu)\in \mathscr Q_F$ and $\lambda^{(n)}$ converges to $\lambda$ in some sense,  then $\bP^{(n)}_{\lambda^{(n)}}$ converges  weakly  to $\bP_\lambda$ on $D_{\overline{\bN}_\partial}[0,\infty)$.  

It should be noted that the assumption of the target process being a Feller $Q$-process, i.e., $(\gamma, \beta, \nu)\in \mathscr Q_F$, appears to be necessary. According to Theorem \ref{THM64}, when $X$ is a Doob process, its semigroup exhibits improved analytic properties when restricted to the proper subspace $\mathbf{C}_{\nu,\gamma}$ of $C(\overline{\bN})$.  Thus,  establishing convergence properties with respect to all functions in $C(\overline{\bN})$ seems challenging.  However, it is worth mentioning that Doob processes and Feller $Q$-processes with $|\nu|<\infty$ are well understood (see \cite{L23}). Therefore, our main focus is on examining the approximation of Feller $Q$-processes with $|\nu|=\infty$. From this perspective, the assumption of $(\gamma,\beta,\nu)\in \mathscr Q_F$ does not result in any loss.

\subsection{Approximating triples}

We begin by examining condition (1c) in Theorem~\ref{THM42}.  According to \eqref{eq:214},  this condition can be expressed in terms of the triples in $\mathscr Q$ as
\begin{equation}\label{eq:51}
\begin{aligned}
	\lim_{n\rightarrow \infty}&\frac{\nu^{(n)}\left(R^\text{min}_\alpha f\right)+\beta^{(n)}\cdot \mu\left(u^\text{min}_\alpha f\right)}{\gamma^{(n)}+\nu^{(n)}\left(1-u^\text{min}_\alpha\right)+\beta^{(n)}\alpha\cdot \mu\left( u^\text{min}_\alpha\right)} =\frac{\nu\left(R^\text{min}_\alpha f\right)+\beta\cdot \mu\left(u^\text{min}_\alpha f\right)}{\gamma+\nu\left(1-u^\text{min}_\alpha\right)+\beta\alpha \cdot \mu\left( u^\text{min}_\alpha\right)}.
\end{aligned}\end{equation}
The verification of the convergence of the corresponding parts in \eqref{eq:51} is a straightforward process in establishing the validity of (1c). Specifically, we can examine the convergence properties of the triples $(\gamma^{(n)}, \beta^{(n)}, \nu^{(n)})$ as follows.

\begin{definition}\label{DEF51}
The triple $(\gamma^{(n)}, \beta^{(n)}, \nu^{(n)})\in \mathscr Q$ is said to converge to $(\gamma,\beta,\nu)\in \mathscr Q$,  if
\[
	\lim_{n\rightarrow \infty}\gamma^{(n)}=\gamma,\quad \lim_{n\rightarrow \infty}\beta^{(n)}=\beta,\quad \lim_{n\rightarrow \infty}\nu^{(n)}_k=\nu_k,\; \forall k\in \bN,
\]
and 
\begin{equation}\label{eq:52}
	\lim_{n\rightarrow \infty} \nu^{(n)}\left(1-u^\text{min}_\alpha\right)=\nu\left(1-u^\text{min}_\alpha\right),\quad \forall \alpha>0\; (\text{or equivalently, }\exists \alpha>0).  
\end{equation}
\end{definition}
\begin{remark}
The pointwise convergence $\lim_{n\rightarrow \infty}\nu^{(n)}_k=\nu_k, \forall k\in \bN,$ is equivalent to the \emph{vague convergence} of the measure $\nu^{(n)}$ on $\mathbb{N}$ to $\nu$,  i.e., $\nu^{(n)}(f)\rightarrow \nu(f)$ for any $f\in C_c(\mathbb{N})$. However,  this condition alone is not sufficient to guarantee \eqref{eq:52},  because $1-u_\alpha^\text{min}\in C_0(\mathbb{N})$ and $\nu$ or $\nu^{(n)}$ may be infinite measure.  The equivalence between the two formulations in condition \eqref{eq:52} can be established by using the following two inequalities and the generalized dominated convergence theorem (similar to the approach in the proof of Lemma~\ref{LM53}): $1-u_{\alpha_1}^\text{min}(k)\leq 1-u_{\alpha_2}^\text{min}(k)$ and $1-u_{2\alpha_1}^\text{min}(k)\leq 2(1-u_{\alpha_1}^\text{min}(k))$ for all $\alpha_2> \alpha_1>0$.
\end{remark}

When $\nu^{(n)}$ monotonically converges pointwise to $\nu$, i.e., $\nu^{(n)}_k\uparrow \nu_k$ (or $\nu^{(n)}_k\downarrow \nu_k$) for all $k\in \mathbb{N}$,  condition \eqref{eq:52} is clearly satisfied by the dominated convergence theorem.  This case includes approximation sequence obtained using the simplest truncation method:
\[
\gamma^{(n)}:=\gamma,\quad \beta^{(n)}:=\beta,\quad \nu^{(n)}_k:=\nu_k,\;0\leq k\leq n,\quad \nu^{(n)}_k:=0,\;k>n,
\]	
which serves as the primary motivation for the study of this section.  However, Definition~\ref{DEF51} is not restricted to the scenario described in this example, where a sequence of simple processes approximates a complex process. It also allows for the opposite situation. For instance, let us consider the case where $\infty$ is a regular boundary and $(\gamma, \beta, \nu)\in \mathscr Q$ satisfies $\beta>0$ and $|\nu|=\infty$.  Define
\[
	\gamma^{(n)}:=\gamma, \quad \beta^{(n)}:=\beta, \quad \nu^{(n)}_k:=0,\;0\leq k\leq n,\quad \nu^{(n)}_k:=\nu_k,\; k>n.
\]
Then,  $|\nu^{(n)}|=\infty$ for all $n\in \bN$,  while $\left(\gamma^{(n)},\beta^{(n)},\nu^{(n)}\right)$ converges to $(\gamma, \beta,0)$ in the sense of Definition~\ref{DEF51}.  

It is easy to prove that the convergence based on the triple, as defined above, implies \eqref{eq:51}.

\begin{lemma}\label{LM53}
Assume that $(\gamma^{(n)}, \beta^{(n)}, \nu^{(n)})$ converges to $(\gamma,\beta,\nu)$ in the sense of Definition~\ref{DEF51}.  Then all the conditions in Theorem~\ref{THM42} are satisfied.
\end{lemma}
\begin{proof}
It suffices to prove $\nu^{(n)}(R_\alpha^\text{min} f)\rightarrow \nu(R_\alpha^\text{min}f)$ for all $f\in C(\overline{\bN})$.  To do this,  note that
\[
\nu^{(n)}_k  R_\alpha^\text{min} f(k)\rightarrow \nu_k  R_\alpha^\text{min} f(k),\quad \nu^{(n)}_k(1-u^\text{min}_\alpha(k))\rightarrow \nu_k(1-u^\text{min}_\alpha(k))
\]
and 
\[
	|\nu^{(n)}_k R^\text{min}_\alpha f(k)|\leq \frac{\|f\|_\infty}{\alpha}\nu^{(n)}_k(1-u_\alpha^\text{min}(k))
\]
for fixed $k\in \bN$.  Thus,  $\nu^{(n)}(R_\alpha^\text{min} f)\rightarrow \nu(R_\alpha^\text{min}f)$ follows from \eqref{eq:52} by utilizing the generalized dominated convergence theorem (see \cite[\S2.3, Exercise 20]{F99}).
\end{proof}

\subsection{Convergence in finite-dimensional distributions}

For each $n$,  define $\bC^{(n)}:=C(\overline{\bN})$ if $X^{(n)}$ is a Feller $Q$-process,  and define $\bC^{(n)}:=\bC_{\nu^{(n)},\gamma^{(n)}}$ as \eqref{eq:611-3} (with $\nu=\nu^{(n)},\gamma=\gamma^{(n)}$) if $X^{(n)}$ is a Doob process.  According to \cite[Theorem~6.3]{L23} or Theorem~\ref{THM64},  the transition semigroup $(P^{(n)}_t)_{t\geq 0}$ acts on $\bC^{(n)}$ as a strongly continuous contractive semigroup.  Let $\sL^{(n)}$ with domain $\cD(\sL^{(n)})$ denote the infinitesimal generator of $(P^{(n)}_t)_{t\geq 0}$ on $\bC^{(n)}$.  Similarly,  we define $\bC:=C(\overline{\bN})$ if $X$ is a Feller $Q$-process,  and $\bC:=\bC_{\nu,\gamma}$ if $X$ is a Doob process.  Denote by $\sL$ with domain $\cD(\sL)$ the infinitesimal generator of $(P_t)_{t\geq 0}$ on $\bC$.  

\begin{lemma}\label{LM54}
Assume that $(\gamma^{(n)}, \beta^{(n)}, \nu^{(n)})$ converges to $(\gamma,\beta,\nu)$ in the sense of Definition~\ref{DEF51}.  Then the following conclusions hold:
\begin{itemize}
\item[(1)] For any $g\in \bC$,  there exists a sequence $g_n\in \bC^{(n)}$ such that 
\begin{equation}\label{eq:98}
\|g_n-g\|_\infty=0,\quad n\rightarrow \infty.
\end{equation}
\item[(2)] For any $f\in \cD(\sL)$,  there exists a sequence $f_n\in \cD(\sL^{(n)})$ such that 
\begin{equation}\label{eq:914}
	\|f_n-f\|_\infty+\|\sL^{(n)} f_n-\sL f\|_\infty\rightarrow 0,\quad n\rightarrow \infty.  
\end{equation}
\end{itemize}
\end{lemma}
\begin{proof}
The sequence of processes $X^{(n)}$ can be divided into at most two subsequences: one consisting of Feller $Q$-processes and the other consisting of Doob processes. Thus, it is sufficient to consider each subsequence separately. First, we will examine the subsequence,  still denoted by $X^{(n)}$,  consisting of Feller $Q$-processes.  Note that $\mathbf{C}_n = C(\overline{\mathbb{N}})$ for all $n\in \bN$.  Hence, the existence of $g_n$ is evident.  For $f=R_\alpha h\in \cD(\sL)$ with $h\in \bC$,  we can take $f_n:=R^{(n)}_\alpha h \in \cD(\sL^{(n)})$,  which satisfies \eqref{eq:914} by (1b) of Theorem~\ref{THM42} and the Hille-Yosida theorem.  

It remains to consider the subsequence consisting of Doob processes.  From now on, we assume that all $X^{(n)}$ are Doob processes.  According to Definition~\ref{DEF51},  we have $\beta=\lim_{n\rightarrow \infty}\beta^{(n)}=0$. 

 (1) Assume without loss of generality that $g=R_\alpha h\in \cD(\sL)$ with $h\in \bC$.  According to \eqref{eq:21-2} or Theorem~\ref{THM64}, we have
 \begin{equation}\label{eq:55}
 	\sum_{k\in \bN} (g(\infty)-g(k))\nu_k +\gamma g(\infty)=0.
 \end{equation}
 Note that $|\nu|>0$ because of the condition \eqref{eq:B1}.   Let $N\in \bN$ such that $\sum_{k=0}^N\nu_k>0$,  and assume without loss of generality that $\sum_{k=0}^N\nu^{(n)}_k>0$ for all $n\geq 1$.  Define
 \[
 	g_n:=g+\chi_n \cdot 1_{\{0,1,\cdots, N\}},
 \]
 where $$\chi_n:=\frac{\sum_{k\in \bN}(g(\infty)-g(k))\nu^{(n)}_k+\gamma^{(n)}g(\infty)}{\sum_{k=0}^N\nu^{(n)}_k}.$$  We explain why $\sum_{k\in \bN}(g(\infty)-g(k))\nu^{(n)}_k$ is finite in the definition of $\chi_n$.  In fact,  according to \eqref{eq:215},  we have
 \[
 	g(\infty)-g(k)=R_\alpha h(\infty)-R_\alpha h(k)=R_\alpha h(\infty)(1-u^\text{min}_\alpha(k))-R^\text{min}_\alpha h(k). 
 \]
 Thus,
 \begin{equation}\label{eq:56-2}
 	\sum_{k\in \bN}(g(\infty)-g(k))\nu^{(n)}_k=R_\alpha h(\infty)\cdot \nu^{(n)}(1-u^\text{min}_\alpha)-\nu^{(n)}(R_\alpha^\text{min}h),
 \end{equation}
 which is finite due to \eqref{eq:210}.  In addition,  it follows from \eqref{eq:52} and Lemma~\ref{LM53} that
 \[
 \begin{aligned}
 	\lim_{n\rightarrow \infty}\sum_{k\in \bN}(g(\infty)-g(k))\nu^{(n)}_k&=R_\alpha h(\infty)\cdot \nu(1-u^\text{min}_\alpha)-\nu(R_\alpha^\text{min}h) \\
 	&=\sum_{k\in \bN}(g(\infty)-g(k))\nu_k.  
\end{aligned} \]
Hence,  $\chi_n\rightarrow 0$ by \eqref{eq:55},  which yields that $\|g_n-g\|_\infty\rightarrow 0$.  Finally,  it is straightforward to verify that $g_n\in C(\overline{\bN})$ with $\sum_{k\in \bN} (g_n(\infty)-g_n(k))\nu^{(n)}_k +\gamma^{(n)} g_n(\infty)=0$ by using the definition of $\chi_n$.  In other words,  $g_n\in \bC^{(n)}$.  
 
 (2) Let $g:=f-\sL f\in \bC$.  By applying the first assertion to this $g$,  we obtain a sequence $g_n\in \bC^{(n)}$ that satisfies \eqref{eq:98}.  Define $f_n:=R^{(n)}_1 g_n\in \cD(\sL^{(n)})$.  Since $f=R_1g$,  it follows from (1b) of Theorem~\ref{THM42} and \eqref{eq:98} that
 \[
 \begin{aligned}
 	\|f_n-f\|_\infty&\leq \left\|R^{(n)}_1 g_n-R_1^{(n)}g\right\|_\infty+\left\|R^{(n)}_1g-R_1g\right\|_\infty \\
 	&\leq \left\|g_n-g\right\|_\infty+\left\|R^{(n)}_1g-R_1g\right\|_\infty\rightarrow 0. 
 \end{aligned}\] 
 Additionally,  we note that
\[
	\sL^{(n)} f_n=f_n-g_n,\quad \sL f=f-g.  
\]	
Therefore, we can also conclude that $\left\|\sL^{(n)} f_n-\sL f\right\|_\infty\rightarrow 0$.  
\end{proof}

Let $\mathcal{P}(\bN_\partial)$ denote the family of all probability measures on $\bN_\partial$.  A sequence of measures $\lambda^{(n)}\in \mathcal{P}(\bN_\partial)$ is said to converge weakly to $\lambda\in \mathcal{P}(\bN_\partial)$ if $\lim_{n\rightarrow \infty}\lambda^{(n)}(f)=\lambda(f)$ for all $f\in C_b(\bN_\partial)$.  It is worth noting that this weak convergence is equivalent to the vague convergence,  i.e.,  $\lim_{n\rightarrow \infty}\lambda^{(n)}(f)=\lambda(f)$ for all $f\in C_c(\bN_\partial)$; see,  e.g.,  \cite[\S7.3, Exercise~26]{F99}. 

Now we are ready to prove the convergence of $X^{(n)}$ in finite-dimensional distributions to a Feller $Q$-process $X$.  

\begin{theorem}\label{THM55}
Assume that $(\gamma^{(n)}, \beta^{(n)}, \nu^{(n)})$ converges to $(\gamma,\beta,\nu)$ in the sense of Definition~\ref{DEF51},  and that $(\gamma,\beta,\nu)\in \mathscr Q_F$.  Then 
\begin{equation}\label{eq:55-2}
\lim_{n\rightarrow \infty}\sup_{t\in [0,T], k\in \bN_\partial}\left|P^{(n)}_tf(k)-P_tf(k)\right|=0
\end{equation}
 for all $f\in C(\overline{\bN}_\partial)$ and $T\geq 0$.  
Furthermore,  if $\lambda^{(n)}\in \mathcal{P}(\bN_\partial)$ converges weakly to $\lambda\in \mathcal{P}(\bN_\partial)$,  then for any $0= t_1<\cdots<t_N$, and  $f_1,\cdots, f_N\in C(\overline{\bN}_\partial)$ with $N\geq 1$,  
\begin{equation}\label{eq:56}
	\lim_{n\rightarrow \infty}\bE^{(n)}_{\lambda^{(n)}} \left(f_1(X_{t_1})\cdots f_N(X_{t_N}) \right)=\bE_\lambda \left(f_1(X_{t_1})\cdots f_N(X_{t_N}) \right).
\end{equation}
\end{theorem}
\begin{proof}
Consider $f\in \cD(\sL)$, and take a sequence $f_n\in  \cD(\sL^{(n)})$ satisfying \eqref{eq:914}.  We will first apply Lemma~\ref{LM91} with $\mathbf{B}=C(\overline{\bN})$ and $F_n(t):=P^{(n)}_tf_n-P_tf\in \mathbf{B}$ to conclude that
\begin{equation}\label{eq:915-2}
\lim_{n\rightarrow \infty} \sup_{0\leq t\leq T}\left\|P^{(n)}_tf_n-P_tf\right\|_\infty=0.  
\end{equation}
Clearly, the first condition (i) and \eqref{eq:90} in Lemma~\ref{LM91} hold for this $F(t)$.  It suffices to prove the equicontinuity of $F(t)$.  In fact,  the Hille-Yosida theorem indicates that
\[
	F_n(t)-F_n(t')=\int_{t'}^t \left(P^{(n)}_s \sL^{(n)} f_n-P_s \sL f\right)ds.
\]
Since $\|P^{(n)}_s \sL^{(n)} f_n-P_s \sL f\|_\infty\leq \|\sL^{(n)} f_n\|_\infty+\|\sL f\|_\infty \rightarrow 2\|\sL f\|_\infty$,  it is straightforward to obtain the equicontinuity of $F_n(t)$.  Therefore,  \eqref{eq:915-2} is established.  

We are now in a position to demonstrate \eqref{eq:55-2}.  Since $P^{(n)}_t1_{\overline{\bN}_\partial}=P_t1_{\overline{\bN}_\partial}\equiv 1$,  it is sufficient to consider $f\in C(\overline{\bN})$.  
Take an arbitrary $\varepsilon>0$.  We need to show that there exists an integer $N$ such that for any $n>N$,  
\begin{equation}\label{eq:915}
\sup_{t\in [0,T],  k\in \bN}\left|P^{(n)}_t f(k)-P_tf(k)\right|<\varepsilon.
\end{equation}
Since $\bC=C(\overline{\bN})$,  we can take $\tilde{f}\in \cD(\sL)$ such that $\|\tilde{f}-f\|_\infty<\varepsilon/4$.  Let  $\tilde{f}_n\in \cD(\sL^{(n)})$ be the sequence for $\tilde{f}$ as in \eqref{eq:914},  i.e., 
\[
		\left\|\tilde f_n-\tilde f\right\|_\infty+\left\|\sL^{(n)} \tilde f_n-\sL \tilde f\right\|_\infty\rightarrow 0.
\]
It follows from \eqref{eq:915-2} that
\[
\lim_{n\rightarrow \infty} \sup_{0\leq t\leq T}\left\|P^{(n)}_t \tilde f_n-P_t\tilde f\right\|_\infty=0.  
\]
Particularly,  there exists an integer $N$ such that for any $n>N$, 
\begin{equation}\label{eq:917}
	\left\|\tilde f_n-\tilde f\right\|_\infty<\varepsilon/4,\quad \sup_{0\leq t\leq T}\left\|P^{(n)}_t \tilde f_n-P_t\tilde f\right\|_\infty<\varepsilon/4.  
\end{equation}
Note that for any $k\in \bN,  t\in [0,T]$ and $n>N$,  
\[
\begin{aligned}
|P^{(n)}_t f(k)-P_tf(k)|&\leq |P^{(n)}_t f(k)-P^{(n)}_t \tilde{f}(k)| +|P^{(n)}_t \tilde f(k)-P^{(n)}_t \tilde{f}_n(k)| \\&\qquad +|P^{(n)}_t \tilde f_n(k)-P_t \tilde{f}(k)| +|P_t \tilde f(k)-P_t f(k)|  \\
&\leq \|f-\tilde{f}\|_\infty+\|\tilde{f}-\tilde{f}_n\|_\infty+\|P^{(n)}_t \tilde f_n-P_t \tilde{f}\|_\infty+\|\tilde{f}-f\|_\infty.
\end{aligned}\]
By means of $\|\tilde{f}-f\|_\infty<\varepsilon/4$ and \eqref{eq:917},  we can obtain \eqref{eq:915}.  

Before proving \eqref{eq:56}, let us clarify two facts. Firstly, it is not difficult to derive from \eqref{eq:55-2} that 
\begin{equation}\label{eq:511}
\lim_{n\rightarrow \infty}\sup_{t\in [0,T], k\in \bN_\partial}\left|P^{(n)}_tg_n(k)-P_tg(k)\right|=0
\end{equation}
 holds for any functions $g_n,g\in C(\overline{\bN}_\partial)$ satisfying $\|g_n-g\|_\infty \rightarrow 0$. Secondly, let $\varrho_n(\cdot):=\bP^{(n)}_{\lambda^{(n)}}\left(\left(X_{t_1},\cdots, X_{t_N}\right)\in \cdot\right)$ and $\varrho(\cdot):=\bP_{\lambda}\left(\left(X_{t_1},\cdots, X_{t_N}\right)\in \cdot\right)$. It can be shown that $\varrho_n$ and $\varrho$ are probability measures on $\left(\overline{\bN}_\partial\right)^N$ (the $N$-fold product space of $\overline{\bN}_\partial$)  satisfying $\varrho_n\left(\left(\bN_\partial\right)^N\right)=\varrho\left(\left(\bN_\partial\right)^N\right)=1$. Thus, according to \cite[\S7.3, Exercise 26]{F99},  to prove \eqref{eq:56} as required, it is equivalent to proving that $\varrho_n$ converges vaguely to $\varrho$ on $\left(\bN_\partial\right)^N$. In other words, we can assume without loss of generality that $f_1,\cdots, f_N\in C_c(\bN_\partial)$ in \eqref{eq:56}.
 
 By utilizing \eqref{eq:55-2} for $f_N$,  we have
 \begin{equation}\label{eq:513-2}
 	\sup_{k\in \bN_\partial}\left|P^{(n)}_{t_N-t_{N-1}}f_N(k)-P_{t_N-t_{N-1}}f_N(k)\right|\rightarrow 0.
 \end{equation}
 Although $P^{(n)}_{t_N-t_{N-1}}f_N$ may not be in $C(\overline{\bN}_\partial)$,  it holds that
 \[
 	g_{N-1}^{(n)}:=f_{N-1}P^{(n)}_{t_N-t_{N-1}}f_N\in C_c(\bN_\partial),\quad g_{N-1}:=f_{N-1}P_{t_N-t_{N-1}}f_N\in C_c(\bN_\partial)
 \]
because of $f_{N-1}\in C_c(\bN_\partial)$.  Then, \eqref{eq:513-2}  indicates 
 \[
 	\lim_{n\rightarrow \infty}\left\|g_{N-1}^{(n)}-g_{N-1}\right\|_\infty=0.
 	\]
 	It follows from \eqref{eq:511} that
 	\[
 	\begin{aligned}
\sup_{k\in \bN_\partial}&\left|P^{(n)}_{t_{N-1}-t_{N-2}}\left( f_{N-1}P^{(n)}_{t_N-t_{N-1}}f_N\right)-P_{t_{N-1}-t_{N-2}}\left( f_{N-1}P_{t_N-t_{N-1}}f_N\right)\right| \\
&=\sup_{k\in \bN_\partial}\left| P^{(n)}_{t_{N-1}-t_{N-2}}g^{(n)}_{N-1}-P_{t_{N-1}-t_{N-2}}g_{N-1} \right|\rightarrow 0.
 \end{aligned}	\]
 By induction,  we obtain that
 \[
 \begin{aligned}
 	&g^{(n)}_1:=f_1 P^{(n)}_{t_2-t_1}\left(f_2\cdots P^{(n)}_{t_{N-1}-t_{N-2}}\left( f_{N-1}P^{(n)}_{t_N-t_{N-1}}f_N \right)\right)\in C_c(\bN_\partial),  \\
 	&g_1:=f_1 P_{t_2-t_1}\left(f_2\cdots P_{t_{N-1}-t_{N-2}}\left( f_{N-1}P_{t_N-t_{N-1}}f_N \right)\right)\in C_c(\bN_\partial),
 \end{aligned}\]
 and 
 \[
 \lim_{n\rightarrow \infty}\left\|g_{1}^{(n)}-g_{1}\right\|_\infty=0.
 \]
 Since $\lambda^{(n)}$ converges weakly to $\lambda$,  it is easy to verify that $\lambda^{(n)}(g^{(n)}_1)\rightarrow \lambda(g_1)$.  This is precisely the desired \eqref{eq:56}.  
\end{proof}
\begin{remark}
It is worth noting that \eqref{eq:55-2} is stronger than either condition in Theorem~\ref{THM42},  since by the dominated convergence theorem,  \eqref{eq:55-2} implies condition (1a) in Theorem~\ref{THM42}. 
\end{remark}

\subsection{Weak convergence}

Under the assumption stated in Theorem~\ref{THM55}, we proceed to examine the convergence of $X^{(n)}$. As discussed in \S~\ref{SEC26}, the process $X^{(n)}$ with initial distribution $\lambda^{(n)}$ can be represented as a probability measure $\bP^{(n)}_{\lambda^{(n)}}$ on the Skorohod topological space $D_{\overline{\bN}_\partial}[0,\infty)$. Similarly, the process $X$ with initial distribution $\lambda$ can be represented as a probability measure $\bP_{\lambda}$ on $D_{\overline{\bN}_\partial}[0,\infty)$.  We aim to establish the weak convergence of $\bP^{(n)}_{\lambda^{(n)}}$ to $\bP_\lambda$.

\begin{theorem}\label{THM57}
Assume that $(\gamma^{(n)}, \beta^{(n)}, \nu^{(n)})$ converges to $(\gamma,\beta,\nu)\in \mathscr Q_F$ in the sense of Definition~\ref{DEF51}.   If $\lambda^{(n)}\in \mathcal{P}(\bN_\partial)$ converges weakly to $\lambda\in \mathcal{P}(\bN_\partial)$,  then $\bP^{(n)}_{\lambda^{(n)}}$ converges weakly to $\bP_\lambda$ on the Skorohod topological space $D_{\overline{\bN}_\partial}[0,\infty)$, i.e., 
\begin{equation}\label{eq:512}
	\lim_{n\rightarrow \infty}\int_{D_{\overline{\bN}_\partial}[0,\infty)}F(w)\bP^{(n)}_{\lambda^{(n)}}(dw)=\int_{D_{\overline{\bN}_\partial}[0,\infty)}F(w)\bP_{\lambda}(dw)
\end{equation}
for all $F\in C_b\left(D_{\overline{\bN}_\partial}[0,\infty); d \right)$,  where $C_b\left(D_{\overline{\bN}_\partial}[0,\infty);d\right)$ denotes the family of all bounded continuous functions on $D_{\overline{\bN}_\partial}[0,\infty)$ equipped with the Skorohod topology.  
\end{theorem}
\begin{proof}
The sequence of processes $X^{(n)}$ can be divided into at most two subsequences. One subsequence consists of Doob processes, while the other subsequence consists of Feller $Q$-processes. Therefore, it is sufficient to prove \eqref{eq:512} separately for each of these subsequences. For the subsequence consisting of Feller $Q$-processes, we can directly apply the result from \cite[\S4 Theorem~2.5]{EK09}. Further details can be found in the proof of \cite[Theorem~9.2]{L23}. Now, let us focus on the subsequence consisting of Doob processes. Without loss of generality, we assume that all $X^{(n)}$ are Doob processes.


Let $D_{\mathbb{R}^k}[0,\infty)$ denote the Skorohod topological space consisting of all c\`adl\`ag functions on $\mathbb{R}^k$ for $k\geq 1$.  For $g_1,\cdots,  g_k\in C(\overline{\bN}_\partial)$,  the vector-valued function $G:=(g_1,\cdots, g_k)$ induces a Borel measurable map
\[
	G: D_{\overline{\bN}_\partial}[0,\infty)\rightarrow D_{\bR^k}[0,\infty),\quad w\mapsto w_G,
\]
where $w_G(t):=(g_1(w(t)),\cdots, g_k(w(t)))$.  Thus,  the image measures 
\[
	\hat\bP^{(n)}:=\bP^{(n)}_{\lambda^{(n)}}\circ G^{-1},\quad \hat \bP:=\bP_\lambda\circ G^{-1}
\]
are probability measures on $D_{\bR^k}[0,\infty)$.  We aim to show the weak convergence of $\hat\bP^{(n)}$ to $\hat \bP$ on $D_{\bR^k}[0,\infty)$,  and then, \eqref{eq:512} follows by applying \cite[\S3, Corollary~9.2]{EK09}.

We first demonstrate that the sequence $\{\hat\bP^{(n)}: n\geq 1\}$ of probability measures on $D_{\bR^k}[0,\infty)$ is relatively compact.  According to \cite[\S3 Theorem~9.4 and Remark~9.5~(b)]{EK09},  it suffices to consider the case $k=1$ and, for fixed $g=g_1\in C(\overline{\bN}_\partial)$,  $\varepsilon>0$,  and $T>0$, to find a pair of real-valued progressive processes $(Y^{(n)}_t,  Z^{(n)}_t)$ on $(\Omega, \sF^{(n)}, \bP^{(n)}_{\lambda^{(n)}})$,  adapted to the filtration $\sF^{(n)}_t$ of $X^{(n)}$,  for all $n\geq 1$,  satisfying the following conditions:
\begin{itemize}
\item[(i)] $\sup_{t\geq 0}\bE^{(n)}_{\lambda^{(n)}} \left|Y^{(n)}_t\right|,  \sup_{t\geq 0}\bE^{(n)}_{\lambda^{(n)}} \left|Z^{(n)}_t\right|<\infty$, and 
\begin{equation}\label{eq:513}
Y^{(n)}_t-\int_0^t Z^{(n)}_sds
\end{equation}
 is an $\sF^{(n)}_t$-martingale.
\item[(ii)] It holds that
\begin{equation}\label{eq:921}
\limsup_{n\rightarrow \infty}\bE^{(n)}_{\lambda^{(n)}}\left(\sup_{t\in [0,T]}\left|Y^{(n)}_t-g(X_t)\right| \right)<\varepsilon
\end{equation}
and
\begin{equation}\label{eq:922-2}
\limsup_{n\rightarrow \infty} \bE^{(n)}_{\lambda^{(n)}}\left(\sup_{t\in [0,T]} \left|Z^{(n)}_t\right|\right)<\infty. 
\end{equation}
\end{itemize}
In fact,  since $(R_\alpha)_{\alpha>0}$ is strongly continuous on $C(\overline{\bN}_\partial)$,  there exists $\alpha_0>0$ such that 
\begin{equation}\label{eq:923-2}
	\left\|\alpha_0 R_{\alpha_0} g-g\right\|_\infty<\varepsilon.  
\end{equation}
Applying the first statement of Lemma~\ref{LM54} to $g$,  we can obtain a sequence $g_n\in \bC_n$ such that $\|g_n-g\|_\infty\rightarrow 0$.  It follows from (1b) of Theorem~\ref{THM42} that
\begin{equation}\label{eq:924-2}
\lim_{n\rightarrow \infty}\left\|\alpha_0 R^{(n)}_{\alpha_0} g_n -\alpha_0 R_{\alpha_0} g\right\|_\infty =0. 
\end{equation}
For each $n$,  we define
\[
	Y^{(n)}_t:=\alpha_0 R^{(n)}_{\alpha_0} g_n(X_t),\quad Z^{(n)}_t:=\alpha_0 \left(\alpha_0 R^{(n)}_{\alpha_0}g_n-g_n \right)(X_t),
\]
and verify the conditions listed above as follows.  
It is evident that the first part of (i) holds true,  since $\sup_{n\in \bN}\|g_n\|_\infty<\infty$.  To demonstrate that \eqref{eq:513} is an $\sF^{(n)}_t$-martingale,  we note that $f:=\alpha_0 R^{(n)}_{\alpha_0}g_n \in \cD(\sL^{(n)})$ and $\sL^{(n)} f=\alpha_0 \left(\alpha_0 R^{(n)}_{\alpha_0}g_n-g_n \right)$.  Then 
\begin{equation}\label{eq:923}
	Y^{(n)}_t-\int_0^t Z^{(n)}_sds=f(X_t)-\int_0^t \sL^{(n)} f(X_s)ds
\end{equation}
is adapted to $\sF^{(n)}_t$.  By virtue of Theorem~\ref{THM64} and the Hille-Yosida theorem, we have 
\begin{equation}\label{eq:924}
\bE^{(n)}_i\left(f(X_t)-f(X_0)\right)=\bE^{(n)}_i\int_0^t \sL^{(n)} f(X_u)du
\end{equation}
for any $i\in \bN_\partial$ and $t\geq 0$.  
Since $X_t$ takes values in $\bN_\partial$ for all $t\geq 0$,  $\bP^{(n)}_{\lambda^{(n)}}$-a.s.,  it follows from the Markov property and \eqref{eq:924} that for any $0\leq s<t$, 
\[
\begin{aligned}
	\bE^{(n)}_{\lambda^{(n)}}\left(f(X_t)-f(X_s)\bigg|\sF^{(n)}_s  \right)&=\bE^{(n)}_{X_s}\left(f(X_{t-s})-f(X_0)\right)\\
	&=\bE^{(n)}_{\lambda^{(n)}}\left(\int_s^t \sL^{(n)} f(X_u)du\bigg|\sF^{(n)}_s\right).
\end{aligned}\]
As a result,  \eqref{eq:923} is an $\sF^{(n)}_t$-martingale.  Additionally,  according to  \eqref{eq:923-2} and \eqref{eq:924-2},  the left hand side of \eqref{eq:921} is not greater than
\[
\limsup_{n\rightarrow \infty}	\left(\left\|\alpha_0 R^{(n)}_{\alpha_0} g_n-\alpha_0 R_{\alpha_0} g\right\|_\infty+\left\|\alpha_0 R_{\alpha_0}g-g\right\|_\infty\right)<\varepsilon,
\]
and it follows from $\|g_n-g\|_\infty\rightarrow 0$ that the left hand side of \eqref{eq:922-2} is not greater than
\[
\limsup_{n\rightarrow \infty} 2\alpha_0 \|g_n\|_\infty=2\alpha_0 \|g\|_\infty<\infty.  
\]
Therefore,  we have established the existence of $(Y^{(n)}_t,Z^{(n)}_t)$,  thereby  completing the proof of the relative compactness of $\hat \bP^{(n)}$. 

In order to conclude that the weak convergence of $\hat\bP^{(n)}$ to $\hat \bP$,  according to \cite[\S3,  Theorem~7.8]{EK09},  it remains to verify that for $0\leq t_1<\cdots<t_m$ and $h_1,\cdots, h_m\in C_b(\bR^k)$ with $m\geq 1$,  
\begin{equation}\label{eq:927}
	\lim_{n\rightarrow \infty}\bE^{(n)}_{\lambda^{(n)}}\left(h_1(G\left(X_{t_1}\right))\cdots h_m\left(G(X_{t_m}\right)) \right)= \bE_\lambda\left(h_1\left(G(X_{t_1})\right)\cdots h_m\left(G(X_{t_m})\right)\right).
\end{equation}
Since $g_1,\cdots, g_k\in C(\overline{\bN}_\partial)$ and $h_i\in C_b(\bR^k)$ for $1\leq i\leq m$,  it follows that $f_i:=h_i\circ G\in C(\overline{\bN}_\partial)$.  Particularly,  \eqref{eq:927} is the consequence of \eqref{eq:56}.  This completes the proof.
\end{proof}

After proving the theorem mentioned above, it becomes apparent that the truncation method \eqref{eq:12} effectively enables the construction of a sequence of simple $Q$-processes, which converges to the target process with an infinite jumping measure as described in \eqref{eq:512}. Additionally, this theorem allows for the construction of various other types of examples. For instance, let us consider the case where $\infty$ is a regular boundary, and we choose an infinite measure $\nu$ that satisfies \eqref{eq:B1}. Define a sequence of measures $\nu^{(n)}$ as follows:
\[
	\nu^{(n)}_k:=0,\;0\leq k\leq n,\quad \nu^{(n)}_k:=\nu_k,\; k>n.
\]
For any constant $\beta>0$, the triple $(0,\beta, \nu^{(n)})$ corresponds to an honest $Q$-process $X^{(n)}$,  which exhibits complex jumping behavior near the boundary $\infty$, but each jump originating from $\infty$ will only enter states that  are beyond $n$.  As $n \rightarrow \infty$, the jumps of $X^{(n)}$ from $\infty$ into $\mathbb{N}$ become increasingly compressed,  and this sequence of processes converges to the $(Q,1)$-process in the sense of \eqref{eq:512}.

\section{Weak convergence for Wang's approximation}

In his 1958 doctoral thesis (see \cite[Chapter 6]{WY92}), Wang constructed a sequence of honest Doob processes for each honest $Q$-process. These processes are designed to converge to the given $Q$-process in the sense of \eqref{eq:12}. In fact, Wang’s construction remains effective even in the non-honest case, and its convergence is stronger than \eqref{eq:12} as it also ensures the weak convergence on the Skorohod topological space.   In this subsection, these findings will be  explained.

\subsection{Wang's approximation}\label{SEC61}

We begin by introducing a transformation on the trajectories. Let $x(t)$ be a c\`adl\`ag function on $\overline{\bN}_\partial$.  Consider two sequences of positive constants $(\alpha_m)$ and $(\beta_m)$ such that 
\[
	0(=:\beta_0)<\alpha_1\leq \beta_1<\alpha_2\leq \beta_2<\cdots.
\]
(These sequences may consist of finite numbers.) We say that the function $y(t)$ is obtained from $x(t)$ by the $C(\alpha_m,\beta_m)$-transformation if 
\[
\begin{aligned}
	&y(t)=x(t),  &\quad 0\leq t<\alpha_1,\\
	&y(d_m+t)=x(\beta_m+t), &\quad 0\leq t<\alpha_{m+1}-\beta_m,
\end{aligned}\]
where $d_1:=\alpha_1$ and $d_{m+1}:=d_m+(\alpha_{m+1}-\beta_m)$.  Intuitively speaking, the $C(\alpha_m,\beta_m)$-transformation discards the trajectory of $x(t)$ corresponding to the interval $[\alpha_m, \beta_m)$, keeps the segment $[0, \alpha_1)$ unchanged, and shifts the remaining parts to the left, connecting them in the original order without intersection, thereby obtaining a new c\`adl\`ag  trajectory $y(t)$.

Let $X$ be either a Doob process or a Feller $Q$-process with parameters $(\gamma,\beta,\nu)\in \mathscr Q$.  Fix $n\in \bN$.  Define $\eta:=\inf\{t>0:X_{t-}=\infty\}$ and $\sigma^{(n)}:=\inf\{t>0: X_t\in \{0,1,\cdots, n,\partial\}\}$ ($\inf\emptyset:=\infty$).  Then,  we define a sequence of stopping times as follows:
 \[
 \eta^{(n)}_1:=\eta,\quad	\sigma^{(n)}_1:=\inf\{t\geq \eta_1^{(n)}: X_t\in \{0,1,\cdots, n,\partial\}\},
 \]
 and if $\eta^{(n)}_{m-1}, \sigma^{(n)}_{m-1}$ are already defined,  we set
 \[
 	\eta^{(n)}_m:=\inf\{t\geq \sigma^{(n)}_{m-1}: X_{t-}=\infty\}\wedge \zeta
 \]
and 
\[
	\sigma^{(n)}_m:=\inf\{t\geq \eta^{(n)}_m: X_t\in \{0,1,\cdots, n,\partial\}\}.
\]
For $n\in \bN$ and every $\omega\in \Omega$,   by performing the $C(\eta^{(n)}_m(\omega),\sigma^{(n)}_m(\omega))$-transformation on $X_t(\omega)$,  we obtain a new trajectory,  denoted by $X^{(n)}_t(\omega)$. 

By following the approach of \cite[Lemma~7.1]{L24},  we can derive that
\begin{equation}\label{eq:44}
	X^{(n)}:=\left(\Omega, \sF,  X^{(n)}_t,(\bP_x)_{x\in \overline{\bN}_\partial}\right)
\end{equation}
is a Doob process with instantaneous distribution $\pi^{(n)}=\bP_0(X_{\sigma^{(n)}}\in \cdot)$.  Note that the parameters of $X^{(n)}$ are (see,  e.g.,  \cite[\S7.12, Theorem~2]{WY92} or \cite[Lemma~8.1]{L24})
\begin{equation}\label{eq:37}
	\gamma^{(n)}=\gamma,\quad \beta^{(n)}=0
\end{equation}
and 
\begin{equation}\label{eq:38-2}
\begin{aligned}
	&\nu^{(n)}_k=\nu_k,\; 0\leq k\leq n-1,\quad \nu^{(n)}_n=\frac{\frac{\beta}{2}+\sum_{k\geq n}(c_\infty-c_k)\nu_k}{c_\infty-c_n}, \\
	& \nu^{(n)}_k=0,\; k\geq n+1.
\end{aligned}\end{equation}
It is straightforward to verify that this special approximating sequence of triples satisfies the conditions in Definition~\ref{DEF51},  if and only if $\beta=0$ (see also  \cite[\S7.12, Theorem~2]{WY92}).  When $\beta>0$ (which is applicable only for the case where $\infty$ is regular),  it can be obtained that
\begin{equation}\label{eq:64}
	\lim_{n\rightarrow \infty}\nu^{(n)}(1-u^\text{min}_\alpha)=\nu(1-u^\text{min}_\alpha)+\beta\alpha\mu(u^\text{min}_\alpha)
\end{equation}
and
\begin{equation}\label{eq:65}
\lim_{n\rightarrow \infty}\nu^{(n)}(R^\text{min}_\alpha f)=\nu(R^\text{min}_\alpha f)+\beta\mu(fu^\text{min}_\alpha),\quad f\in C(\overline{\bN})
\end{equation}
by virtue of (see,  e.g., \cite[\S7.10, (3) and (9)]{WY92} or \cite[Theorem~8.1]{F59})
\begin{equation}\label{eq:66}
	\lim_{n\rightarrow \infty}\frac{\Phi^\text{min}_{nk}(\alpha)}{c_\infty-c_n}=2u^\text{min}_\alpha(k)\mu_k,\quad 	\lim_{n\rightarrow \infty}\frac{1- u^\text{min}_\alpha(n)}{c_\infty-c_n}=2\alpha \mu(u^\text{min}_\alpha).
\end{equation}
Therefore, whether $\beta=0$ or $\beta>0$, each condition in Theorem~\ref{THM42} holds true for Wang’s approximation.

\subsection{Weak convergence for Skorohod topology}

It is worth noting that the expression \eqref{eq:44} for the Doob process is not the realization given in \eqref{eq:213}.  However, it can be easily proven that the mapping induced by \eqref{eq:44},
\begin{equation}\label{eq:67}
	\mathcal{X}^{(n)}:(\Omega,\sF)\rightarrow \left(D_{\overline{\bN}_\partial}[0,\infty), \mathscr B\left(D_{\overline{\bN}_\partial}[0,\infty);d \right)\right),\quad \omega\mapsto X^{(n)}_\cdot(\omega),
	\end{equation}
is measurable. Therefore, similar to what is stated in \S\ref{SEC26}, given an initial distribution $\lambda_n \in \mathcal{P}(\bN_\partial)$, $X^{(n)}$ can be realized as a probability measure on $D_{\overline{\bN}_\partial}[0,\infty)$, denoted by $\bP^{(n)}_{\lambda_n}:=\bP_{\lambda_n}\circ \left(\mathcal{X}^{(n)}\right)^{-1}$.

Although Wang's approximation may not satisfy the conditions in Definition \ref{DEF51}, we can still prove its weak convergence on the Skorohod topology. This is because, upon examining the proof in \S\ref{SEC4}, we can see that the essential role of Definition \ref{DEF51} is to ensure the first statement of Lemma \ref{LM54} holds, as well as guaranteeing the resolvent convergence. 

\begin{theorem}
Let $X$ be a Feller $Q$-process and $X^{(n)}$ be its approximating sequence of Doob processes given in \eqref{eq:44}.   If $\lambda^{(n)}\in \mathcal{P}(\bN_\partial)$ converges weakly to $\lambda\in \mathcal{P}(\bN_\partial)$,  then $\bP^{(n)}_{\lambda^{(n)}}$ converges weakly to $\bP_\lambda$ on $D_{\overline{\bN}_\partial}[0,\infty)$ equipped with the Skorohod topology, i.e., 
\begin{equation}\label{eq:68}
	\lim_{n\rightarrow \infty}\int_{D_{\overline{\bN}_\partial}[0,\infty)}F(w)\bP^{(n)}_{\lambda^{(n)}}(dw)=\int_{D_{\overline{\bN}_\partial}[0,\infty)}F(w)\bP_{\lambda}(dw)
\end{equation}
for all $F\in C_b\left(D_{\overline{\bN}_\partial}[0,\infty); d \right)$. 
\end{theorem}
\begin{proof}
Our goal is to show that the first statement of Lemma~\ref{LM54} still holds true.  It suffices to examine the case $\beta>0$,  and consider $g=R_\alpha h\in \cD(\sL)$ with $h\in \bC$ satisfying
\[
	\frac{\beta}{2}g^+(\infty)+\sum_{k\in \bN}(g(\infty)-g(k))\nu_k+\gamma g(\infty)=0,
\]
where $g^+(\infty):=\lim_{k\rightarrow \infty}\frac{g(\infty)-g(k)}{c_{\infty}-c_k}$.  Note that the measure $\nu^{(n)}$ is given by \eqref{eq:38-2}. 

In the case where $|\nu|>0$,  we can define $\chi_n$ and $g_n$ by the same method as in the proof of Lemma \ref{LM54}. Note that $g_n\in \bC_n$. Therefore, it is sufficient to prove that $\lim_{n\rightarrow \infty}\chi_n=0$.  In fact,  it follows from \eqref{eq:56-2}, \eqref{eq:64},  and \eqref{eq:65} that
\[
\begin{aligned}
	\lim_{n\rightarrow \infty}&\sum_{k\in \bN}(g(\infty)-g(k))\nu^{(n)}_k\\&=R_\alpha h(\infty)\cdot \nu(1-u^\text{min}_\alpha)+R_\alpha h(\infty)\cdot \beta\alpha \mu(u^\text{min}_\alpha)-\nu(R^\text{min}_\alpha h)-\beta\mu(h u^\text{min}_\alpha)\\
	&=\sum_{k\in \bN}(g(\infty)-g(k))\nu_k+R_\alpha h(\infty)\cdot \beta\alpha \mu(u^\text{min}_\alpha)-\beta\mu(h u^\text{min}_\alpha).  
\end{aligned}\]
By utilizing the generalized dominated convergence theorem (see \cite[\S2.3, Exercise~20]{F99}),  we can deduce from \eqref{eq:66} that
\[
	\left(R^\text{min}_\alpha h\right)^+(\infty)=\lim_{n \rightarrow \infty}\sum_{k\in \bN}\frac{\Phi^\text{min}_{nk}(\alpha)h(k)}{c_n-c_\infty}=-2\mu(u_\alpha h).
\]
Thus,  according to \eqref{eq:215} and \eqref{eq:66},  we have
\[
\begin{aligned}
	g^+(\infty)&=\left(R^\text{min}_\alpha h\right)^+(\infty)+R_\alpha h(\infty)\lim_{n\rightarrow \infty}\frac{1-u^\text{min}_\alpha(n)}{n}\\
	&=-2\mu(u_\alpha^\text{min} h)+2\alpha \mu(u^\text{min}_\alpha)R_\alpha h(\infty).
\end{aligned}\]
Therefore, 
\[
\begin{aligned}
\lim_{n\rightarrow \infty}&\left(\sum_{k\in \bN}(g(\infty)-g(k))\nu^{(n)}_k+\gamma^{(n)}g(\infty)\right)\\
&=\frac{\beta}{2}g^+(\infty)+\sum_{k\in \bN}(g(\infty)-g(k))\nu_k+\gamma g(\infty)=0.
\end{aligned}\]
This establishes $\lim_{n\rightarrow \infty}\chi_n=0$. 

In the case where $|\nu|=0$,  we take a sequence of functions $g_n$ as follows:
\[
	g_n(n):=g(\infty)+\frac{2\gamma}{\beta}g(\infty)\cdot (c_\infty-c_n),\quad g_n(k):=g(k),\; k\neq n. 
\]
It is straightforward to verify that $g_n\in \bC_n$ and $\|g_n-g\|_\infty\rightarrow 0$. 
\end{proof}

According to the Skorohod representation theorem (see \cite[\S3, Theorem~1.8]{EK09}), if \eqref{eq:68} holds, then there exist $D_{\overline{\mathbb{N}}_\partial}[0,\infty)$-valued random variables $\tilde{\mathcal{X}}$ and $\tilde{\mathcal{X}}^{(n)}$ on a probability space $(\tilde{\Omega},\tilde{\mathcal{F}},\tilde{\bP})$, such that $\tilde{\mathcal{X}}$ and $\tilde{\mathcal{X}}^{(n)}$ have distributions $\bP_\lambda$ and $\bP_{\lambda^{(n)}}^{(n)}$, respectively, and $\tilde{\mathcal{X}}^{(n)}$ converges to $\tilde{\mathcal{X}}$,  $\tilde{\bP}$-a.s. If $\lambda_n=\lambda$, then $\mathcal{X}^{(n)}$ defined in \eqref{eq:67} and $\mathcal{X}$ given by \eqref{eq:217} are on the same probability space $(\Omega, \mathcal{F}, \bP_\lambda)$. In this case, it raises the question whether $\left(\mathcal{X}^{(n)},\mathcal{X}\right)$ can be the Skorohod representation of $\left(\bP_{\lambda^{(n)}},\bP_\lambda\right)$? To verify this fact, it is equivalent to show that 
\begin{equation}\label{eq:69-2}
	d\left(\mathcal{X}^{(n)}(\omega),\mathcal{X}(\omega)\right)\rightarrow 0
\end{equation}
holds for $\bP_\lambda$-a.s. $\omega$.  Although \eqref{eq:69} provides a pointwise convergence in time $t$ for $\mathbb{P}_\lambda$-a.s. $\omega$, establishing the convergence in the Skorohod topology seems still challenging. Typically, a sufficient condition for the $d$-convergence \eqref{eq:69-2} is the local uniform convergence with respect to time $t$. However, this condition is not guaranteed by \eqref{eq:69}.


\subsection{Skorohod representation on topology for convergence in measure}

Finally, let us consider another simpler metric $d'$ on $D_{\overline{\bN}_\partial}[0,\infty)$ inducing the topology for convergence in (Lebesgue) measure. Under this metric, the sequence $X^{(n)}$ for Wang's approximation converges  not only weakly but almost surely to $X$.

For $w,w'\in D_{\overline{\bN}_\partial}[0,\infty)$,  define
\[
d'\left(w,w'\right):=\sum_{j=1}^\infty \frac{1}{2^j}\int_0^j \frac{|w(t)-w'(t)|}{1+|w(t)-w'(t)|}dt.
\]	
According to,  e.g.,  \cite[\S2.4, Exercise~32]{F99},  $d'$ is a metric on $D_{\overline{\bN}_\partial}[0,\infty)$,  and additionally,  $d\left(w_n,w\right)\rightarrow 0$ if and only if for any $T>0$,  $(w_n(t))_{0\leq t\leq T}$ converges to $(w(t))_{0\leq t\leq T}$ in (Lebesgue) measure on $[0,T]$.  Note that the Borel $\sigma$-algebra $\mathscr B\left(D_{\overline{\bN}_\partial}[0,\infty);d' \right)$ generated by $d'$ is also identical to $\sigma\{\pi_t:t\geq 0\}$, the $\sigma$-algebra generated by all projection maps $\pi_t(w)=w(t)$; see,  e.g.,  \cite[\S8.6]{CW05}.  Thus,  $\mathscr B\left(D_{\overline{\bN}_\partial}[0,\infty);d' \right)=\mathscr B\left(D_{\overline{\bN}_\partial}[0,\infty);d\right)$,  and we do not need to change the expressions for the maps $\mathcal{X}$ and $\mathcal{X}^{(n)}$ given by \eqref{eq:217} and \eqref{eq:67}.

It should be noted that in the following theorem, we do not require the target $Q$-process $X$ to be a Feller $Q$-process.

\begin{theorem}
Let $X$ be a $Q$-process and $X^{(n)}$ be its approximating sequence of Doob processes given in \eqref{eq:44}.  For any $\lambda\in \mathcal{P}(\bN_\partial)$,  it holds in the sense of $\bP_\lambda$-a.s.  that
\begin{equation}\label{eq:611}
	\lim_{n\rightarrow \infty}d'\left(\mathcal{X}^{(n)},\mathcal{X}\right)=0.
\end{equation}
Furthermore,  if $\lambda^{(n)}\in \mathcal{P}(\bN_\partial)$ converges weakly to $\lambda\in \mathcal{P}(\bN_\partial)$,  then $\bP^{(n)}_{\lambda^{(n)}}$ converges weakly to $\bP_\lambda$ on $D_{\overline{\bN}_\partial}[0,\infty)$ equipped with the topology induced by $d'$, i.e., 
\begin{equation}\label{eq:610}
	\lim_{n\rightarrow \infty}\int_{D_{\overline{\bN}_\partial}[0,\infty)}F(w)\bP^{(n)}_{\lambda^{(n)}}(dw)=\int_{D_{\overline{\bN}_\partial}[0,\infty)}F(w)\bP_{\lambda}(dw)
\end{equation}
for all $F\in C_b\left(D_{\overline{\bN}_\partial}[0,\infty); d' \right)$,  where $C_b\left(D_{\overline{\bN}_\partial}[0,\infty);d'\right)$ denotes the family of all bounded continuous functions on $D_{\overline{\bN}_\partial}[0,\infty)$ with respect to the metric $d'$.
\end{theorem}
\begin{proof}
As established in \cite[\S6.4]{WY92} (the non-honest case is examined in \cite[Theorem~A.1]{L24}),  the following convergence holds true: for any $i\in \bN_\partial$,
\begin{equation}\label{eq:69}
\bP_i\left(\lim_{n\rightarrow \infty}X^{(n)}_t=X_t,\forall t\geq 0\right)=1.
\end{equation}
Thus,  for $\bP_\lambda$-a.s.  $\omega\in \Omega$ and any $T>0$,  $(X^{(n)}_t(\omega))_{0\leq t\leq T}$ converges to $(X_t(\omega))_{0\leq t\leq T}$ pointwise in $t$.  This convergence clearly implies the convergence in (Lebesgue) measure on $[0,T]$.  Therefore,  \eqref{eq:611} is established.  

In order to prove \eqref{eq:610},  we define for $i\in \overline{\bN}_\partial$,  
\[
g^{(n)}(i):=\int_{D_{\overline{\bN}_\partial}[0,\infty)}F(w)\bP^{(n)}_{i}(dw)=\int_{\Omega}F\left(\mathcal{X}^{(n)}(\omega)\right)\bP_i(d\omega)
\]
and 
\[
	g(i):=\int_{D_{\overline{\bN}_\partial}[0,\infty)}F(w)\bP_{i}(dw)=\int_{\Omega}F\left(\mathcal{X}(\omega)\right)\bP_i(d\omega).
\]
By utilizing \eqref{eq:611} and the dominated convergence theorem,  we have $\lim_{n\rightarrow \infty}g^{(n)}(i)=g(i)$ for any $i\in \bN_\partial$.  It follows from \cite[Theorem~8.12]{CW05} that $g^{(n)},g\in C(\overline{\bN}_\partial)$.  Thus,  $\lim_{n\rightarrow}\lambda^{(n)}(g)=\lambda(g)$.
To obtain \eqref{eq:610},  it remains to show
\[
	\lim_{n\rightarrow \infty}\left(\lambda^{(n)}(g^{(n)})-\lambda^{(n)}(g)\right)=0.
\]
In fact,  we have
\[
	\lambda^{(n)}_k(g^{(n)}(k)-g(k))\rightarrow 0,\quad k\in \bN_\partial, 
\]
and
\[
\left|\lambda^{(n)}_k(g^{(n)}(k)-g(k))\right|\leq 2\|F\|_\infty \lambda^{(n)}_k\rightarrow 2\|F\|_\infty \lambda_k,\quad k\in \bN_\partial,
\]
where $\|F\|_\infty:=\sup_{w\in D_{\overline{\bN}_\partial}[0,\infty)}|F(w)|$.  Therefore,  we can apply the generalized dominated convergence theorem (see \cite[\S2.3, Exercise~20]{F99}) to obtain
\[
	\lim_{n\rightarrow \infty}\left(\lambda^{(n)}(g^{(n)})-\lambda^{(n)}(g)\right)=\lim_{n\rightarrow \infty}\sum_{k\in \bN_\partial}\lambda^{(n)}_k(g^{(n)}(k)-g(k))=0.
\]
This completes the proof of \eqref{eq:610}. 
\end{proof}


\bibliographystyle{abbrv}
\bibliography{ApproxBD}

\end{document}